\documentclass{amsart} 

\usepackage{amsmath,amsthm}     
\usepackage{amssymb,empheq}            
\usepackage{euscript}           
\usepackage{graphicx,enumerate,calc,lscape,color}
\definecolor{red}{rgb}{1,0,0}
\usepackage[matrix,arrow,curve,frame]{xy}    
\usepackage{hyperref}

\xymatrixcolsep{1.9pc}                          
\xymatrixrowsep{1.9pc}
\newdir{ >}{{}*!/-5pt/\dir{>}}                  

\newtheorem{thm}[subsection]{Theorem}

\newtheorem{prop}[subsection]{Proposition}

\newtheorem{cor}[subsection]{Corollary}
\newtheorem{lemma}[subsection]{Lemma}

\theoremstyle{definition}  
\newtheorem{ex}[subsection]{Example}

\newtheorem{remark}[subsection]{Remark}

  {\end{list}}

\newcommand{\dfn}{\textbf} 

\newcommand{\mdfn}[1]{\dfn{\mathversion{bold}#1}} 

\newcommand{\tens}              {\otimes}               
\newcommand{\iso}               {\cong}

\newcommand{\cat}{\EuScript}    

\newcommand{\cO}{{\cat O}}

\newcommand{\cU}{{\cat U}}

\DeclareMathOperator{\Mon}{Mon}

\newcommand{\field}[1]  {\mathbb #1} 
\newcommand{\A}         {\field A}
\newcommand{\R}         {\field R}

\newcommand{\Z}         {\field Z}
\newcommand{\C}         {\field C}
\newcommand{\M}         {\field M}

\DeclareMathOperator{\rank}{rank}
\DeclareMathOperator{\im}{im}

\newcommand{\ra}{\rightarrow}                   
\newcommand{\lra}{\longrightarrow}              
\newcommand{\llra}[1]{\stackrel{#1}{\lra}}      

\newcommand{\fib}{\twoheadrightarrow}           

\newcommand{\inc}{\hookrightarrow}              

\newcommand{\blank}{-}                          
\newcommand{\bblank}{\und{\ \ }\,}
\newcommand{\und}{\underline}

\newcommand{\he}{\simeq}

\newcommand{\pt}{pt}

\newcommand{\rea}[1]{|{#1}|}             

\newcommand{\ceck}[1]{\Cech(#1)}         
\newcommand{\oceck}[1]{\Cech^{o}(#1)}    
\newcommand{\oreal}[1]{\rea{\oceck{U}}}  
\newcommand{\creal}[1]{\rea{\ceck{U}}}   

\newcommand{\Cech}{\check{C}}

\DeclareMathOperator{\Gr}{Gr}
\DeclareMathOperator{\OGr}{OGr}

\newcommand{\F}{\mathbb{F}}

\DeclareMathOperator{\partition}{part}
\newcommand{\prt}[3]{\partition_{#1,\leq #2}[#3]}

\numberwithin{equation}{subsection}

\newenvironment{myequation}
  {\addtocounter{subsection}{1}\begin{eqnarray}}
  {\end{eqnarray}$\!\!$}

\newcommand{\nosee}[1]{}

\newcommand{\RP}{\R P}
\newcommand{\Ztm}{(\Z/2)_m}

\newcommand{\Inv}{{\mathcal{I}}nv}
\newcommand{\wc}{wc}

\begin{document}

\title{Bigraded cohomology of $\Z/2$-equivariant Grassmannians}

\author{Daniel Dugger}
\address{Department of Mathematics\\ University of Oregon\\ Eugene, OR
97403} 

\email{ddugger@uoregon.edu}


\maketitle

\tableofcontents

\section{Introduction}
Let $\R$ and $\R_-$ denote the two representations of $\Z/2$ on
the real line: the first has the trivial action, the second has the
sign action.  Let $\cU$ denote the infinite direct sum
\[ \cU=\R\oplus \R_{-} \oplus \R \oplus \R_- \oplus \cdots \]
The subjects of this paper are the infinite Grassmannians
$\Gr_k(\cU)$, regarded as spaces with a $\Z/2$-action.
Our goal is to compute the $RO(\Z/2)$-graded cohomology rings
$H^*(\Gr_k(\cU);\Ztm)$, where $\Ztm$ denotes the constant-coefficient
Mackey functor.   These cohomology rings are a notion of equivariant
cohomology that is finer than the classical Borel theory.

Of course our results may be interpreted as giving a calculation of all
characteristic classes, with values in the theory
$H^*(\blank;\Ztm)$, for rank $k$ equivariant bundles.  Previous work
on related problems has been done by Ferland and Lewis \cite{FL} and
by Kronholm \cite{K1,K2}, but the present paper provides the first
complete computation for any single value of $k$ larger than $1$.

\medskip

The rest of this introduction aims to describe the results of the
computation.  The context throughout the paper is the category of
$\Z/2$-spaces, with equivariant maps.  Unless stated otherwise all
spaces and maps are in this category.  

The theory $H^*(\blank;\Ztm)$ is graded by the representation ring
$RO(\Z/2)$.  That is to say, if $V$ is a virtual representation then
the theory yields groups $H^V(\blank;\Ztm)$.  For the group $\Z/2$
every representation has the form $\R^p\oplus (\R_-)^q$ for some $p$
and $q$, and this implies that we may regard our cohomology theory as
being bigraded.  Different authors use different indexing conventions,
but we will use the ``motivic'' indexing described as follows.  The
representation $V=\R^p\oplus (\R_-)^q$ is denoted $\R^{p+q,q}$, and the
corresponding cohomology groups $H^V(\blank;\Ztm)$ will be denoted
$H^{p+q,q}(\blank;\Ztm)$.  In this indexing system the first index is
called the \dfn{topological degree} and the second is called the
\dfn{weight}.  One appeal of this system is that dropping the second
index will always give statements that seem familiar from non-equivariant
topology.  

Before continuing, for ease of reading we will  just write
$\Z/2$ instead of $\Ztm$ in coefficients of cohomology groups.  In the
presence of the bigrading this will never lead to any confusion.  

Let $\M_2$ be the bigraded ring $H^{*,*}(\pt;\Z/2)$, the cohomology
ring of a point.  This is the ground ring of our theory; for any
$\Z/2$-space $X$, the ring $H^{*,*}(X;\Z/2)$ is an algebra over
$\M_2$.  A complete description of $\M_2$ is given in the next
section, but for now one only needs to know that there are special
elements $\tau\in \M_2^{0,1}$ and $\rho\in \M_2^{1,1}$.  

The cohomology ring of the projective space $\Gr_1(\cU)$ has been
known for a while; the motivic analog was computed by Voevodsky, and
the same proof works in the $\Z/2$-equivariant setting.  A careful
proof is written down in \cite[Theorem 4.2]{K2}.  There is an isomorphism of algebras
$H^{*,*}(\Gr_1(\cU);\Z/2)\iso \M_2[a,b]/(a^2=\rho a+\tau b)$ where 
$a$ has bidegree $(1,1)$ and $b$ has bidegree $(2,1)$.  
In non-equivariant topology one has $\rho=0$ and $\tau=1$, so that the
above relation becomes $a^2=b$ and we simply have a polynomial algebra
in a variable of degree $1$---the familiar answer for the mod $2$
cohomology of real projective space.  

Note that additively, $H^{*,*}(\Gr_1(\cU);\Z/2)$ is a free module over
$\M_2$ on generators of the following bidegrees:
\[ (0,0),(1,1),(2,1),(3,2),(4,2),(5,3),(6,3),(7,4),\ldots
\]
corresponding to the monomials $1,a,b,ab,b^2,ab^2,b^3,ab^3,\ldots$
If one forgets the weights, then one gets the degrees for elements in an additive basis
for the singular cohomology $H^*(\RP^\infty;\Z/2)$.  So in this case
one can obtain the equivariant cohomology groups by taking a basis for
the singular cohomology groups, adding appropriate weights, and
changing every $\Z/2$ into a copy of $\M_2$.  We mention this because
it is a theorem of Kronholm \cite{K1} that the same is true in the case of
$\Gr_k(\cU)$ (and for many other spaces as well, though not all
spaces).   Because we know the singular cohomology groups
$H^*(\Gr_k(\R^\infty);\Z/2)$, computing the equivariant version
becomes only a question of knowing what weights to attach to the
generators.  While it might seem that it should be simple to resolve this,
the question has been very resistant until now; the present paper
provides an answer.

\medskip

To state our main results, begin by considering the map
\[ \eta\colon \Gr_1(\cU)\times \cdots \times \Gr_1(\cU) \lra \Gr_k(\cU) \]
that classifies the $k$-fold direct sum of line bundles.  Using the
K\"unneth Theorem, the induced
map on cohomology gives
\[ \eta^*\colon H^{*,*}(\Gr_k(\cU);\Z/2) \ra
H^{*,*}(\Gr_1(\cU);\Z/2)^{\tens k}. \]
Since permuting the factors in a $k$-fold sum yields an isomorphic
bundle,  the image of $\eta^*$ lies in the ring of invariants under
the action of the symmetric group $\Sigma_k$.  That is to say, we may
regard $\eta^*$ as a map 
\begin{myequation}
\label{eq:main}
 \eta^*\colon H^{*,*}(\Gr_k(\cU);\Z/2) \ra \bigl
[H^{*,*}(\Gr_1(\cU);\Z/2)^{\tens k }\bigr ]^{\Sigma_k}. 
\end{myequation}

The first of our results is the following:

\begin{thm}
\label{th:main1}
The map in (\ref{eq:main}) is an isomorphism of bigraded rings.  
\end{thm}

This is the direct analog of what happens in the nonequivariant case.
Let us note, however, that until now neither injectivity nor
surjectivity has been known in the present context.  It must be
admitted up front that in some ways our proof of
Theorem~\ref{th:main1} is not very satisfying:
it does not give any reason, based on first principles, why $\eta^*$
should be an isomorphism.  Rather, the proof proceeds by computing the
codomain of $\eta^*$ explicitly and then running a complicated
spectral sequence for computing the domain of $\eta^*$.  By comparing
what is happening on the two sides, and appealing to the
nonequivariant result at key moments, one can see that there is no
choice but for the map to be an isomorphism---even without resolving
all the differentials in the spectral sequence (of which there are
infinitely many).  The argument is somewhat sneaky, but not terribly
difficult in the end.  However, it depends on a key result proven by
Kronholm \cite{K1} that describes the kind of phenomena that take place inside
the spectral sequence.

The proof of Theorem~\ref{th:main1} is the main component of this
paper.  It is completed in Section~\ref{se:proof}.  Subsequent
sections explore some auxilliary issues, that we describe next.  

\begin{remark}
In non-equivariant topology there are several familiar techniques for
proving Theorem~\ref{th:main1}, perhaps the most familiar being use of
the Serre spectral sequence.  Since the theorem is really about the
identification of characteristic classes, another method that  comes to
mind is the Grothendieck approach to characteristic classes via the
cohomology of projective bundles.  The equivariant analogs of both
these approaches have been partially explored by Kronholm
\cite{K2}, but one runs into a fundamental problem: such calculations
require the use of local coefficient systems, because the fixed sets
of Grassmannians are disconnected.  So they involve a level of
diffculty that is far beyond what happens in the non-equivariant case,
and to date no one has gotten these approaches to work.
Cohomology with local coefficients has been little-explored in the
equivariant setting, but see \cite{Sh} for work in this direction.
\end{remark}

\vspace{0.1in}

To access the full power of Theorem~\ref{th:main1} one should compute
the ring of invariants $[H^{*,*}(\Gr_1(\cU);\Z/2)^{\tens k }\bigr
]^{\Sigma_k}$, which is a purely algebraic problem.  The proof of
Theorem~\ref{th:main1} only requires understanding an additive basis
for this ring.  The second part of the paper examines the multiplicative structure.

In regards to the additive basis, we can state one form of our results
as follows.  Recall that a basis for $H^n(\Gr_k(\R^\infty);\Z/2)$ is
provided by the Schubert cells of dimension $n$, and these are in
bijective correspondence with partitions of $n$ into $\leq k$ pieces.
For example, a basis for $H^6(\Gr_3(\R^\infty);\Z/2)$ is in bijective
correspondence with the set of partitions
\[ [6],\quad [51],\quad [42],\quad [411], \quad [33], \quad [321],\quad [222].
\]
For any such partition $\sigma=[j_1j_2\ldots j_k]$, define its {\it weight\/}
to be 
\[ w(\sigma)=\sum \lceil \tfrac{j_i}{2}\rceil.
\]
So the list of the above seven partitions
have corresponding weights $3,4,3,4,4,4,3$.   
Using this notion,
the following result shows how to write down an $\M_2$-basis for
$H^{*,*}(\Gr_k(\cU);\Z/2)$:

\begin{thm}
$H^{*,*}(\Gr_k(\cU);\Z/2)$ is a free module over $\M_2$ with a basis
$S$, where the elements of $S$ in topological degree
$n$ are in bijective correspondence with partitions of $n$ into at
most $k$ pieces.  This bijection sends a partition $\sigma$ to a basis
element of bidegree $(n,w(\sigma))$ where $w(\sigma)$ is the weight of $\sigma$.
\end{thm}

It is easy to see that for a partition $\sigma$ of $n$ the weight is
also equal to
\[ w(\sigma)=\tfrac{1}{2}\bigl (n+(\text{\# of odd pieces in
$\sigma$})\bigr ).
\]
Using this description we can reinterpret the theorem as follows:

\begin{cor}
The number of free generators for $H^{*,*}(\Gr_k(\cU);\Z/2)$ having
bidegree $(p,q)$ coincides with the number of partitions of $p$ into
at most $k$ pieces where exactly $2q-p$ of the pieces are odd.
\end{cor}

For example, in $H^{7,*}(\Gr_5(\cU);\Z/2)$ we have basis elements in
weights $4$, $5$, and $6$, corresponding to the partitions
\begin{align*}
[7], [61], [52], [43], [421], [322], [2221] \quad &\text{(weight
$4$/one odd piece)}
\\
[511], [4111], [331], [3211], [22111] \quad &\text{(weight $5$/three odd
pieces)} \\
[31111] \quad &\text{(weight $6$/five odd pieces)}.
\end{align*}

\vspace{0.1in}

We next describe a little about the ring structure.  Unlike what
happens in nonequivariant topology, it is not easy to write down a
simple description of the ring of invariants in terms of generators
and relations---except for small values of $k$.  In essense, the
innocuous-looking relation ``$a^2=\rho a+\tau b$'' propogates itself
viciously into the ring of invariants, leading to some unpleasant
bookkeeping.  However, we are able to give a minimal set of
generators for the algebra, and we investigate the relations in low
dimensions.

First, for $1\leq i\leq k$ there are special classes $w_i\in
H^{i,i}(\Gr_k(\cU);\Z/2)$ that we call \dfn{Stiefel-Whitney classes};
they correspond to the usual Stiefel-Whitney classes in singular
cohomology.  There are also special classes $c_i\in
H^{2i,i}(\Gr_k(\cU);\Z/2)$ that we call
\dfn{Chern classes}; their images in non-equivariant cohomology correspond to the mod $2$ reductions of the
usual Chern classes of the complexification of a bundle.  In some
sense these constitute the ``obvious'' characteristic classes that one
might expect.  It is not true, however, that these generate
$H^{*,*}(\Gr_k(\cU);\Z/2)$ as an algebra.  This is easy to explain in
terms of the ring of invariants.  There are two sets of variables
$a_1,\ldots,a_k$ and $b_1,\ldots,b_k$, with $\Sigma_k$ acting on each
as permuation of the indices.  The class $w_i$ is the $i$th
elementary symmetric function in the $a$'s, and likewise $c_i$ is the
elementary symmetric function in the $b$'s.  But there are many other
invariants, for example $a_1b_1+\cdots+a_kb_k$.  

We let $w_j^{(e)}$ be the characteristic class corresponding to the
invariant element $\sum a_{i_1}\ldots a_{i_j}b_{i_1}^e\ldots b_{i_j}^e$.  
Note that $w_j^{(0)}=w_j$.  This particular choice of invariants is
not the only natural one, but it seems to be convenient in a number of
ways.  Among other things, these characteristic classes satisfy a
Whitney formula
\[ w_j^{(e)}(E\oplus F)=\sum_r w_{r}^{(e)}(E)\cdot w_{j-r}^{(e)}(F).
\]
Using the classes $w_j^{(e)}$ we can write down a minimal set of
algebra generators for $H^{*,*}(\Gr_k(\cU);\Z/2)$:

\begin{prop}
The indecomposables of $H^{*,*}(\Gr_k(\cU);\Z/2)$ are represented by
$c_1,\ldots,c_k$ together with the classes $w_{2^i}^{(e)}$ for $1\leq
2^i \leq k$ and $0\leq e \leq \frac{k}{2^i}-1$.  
\end{prop}

Note that the above result gives a slight surprise when $e=0$.  The equivariant
Stiefel-Whitney classes $w_i$ are indecomposable only when $i$ is a
power of $2$.  This phenomenon is familiar in a slightly different
(but related) context---see \cite[Remark 3.4]{M}.  

In practice it is unwieldy to write down a complete set of relations
for $H^{*,*}(\Gr_k(\cU);\Z/2)$.  To give a sense of this, however, we
do it here for $k=2$:

\begin{prop}
\label{pr:k=2}
The algebra $H^{*,*}(\Gr_2(\cU);\Z/2)$ is the quotient of
the ring $\M_2[c_1,c_2,w_1,w_2,w_1^{(1)}]$ by the following relations:
\begin{itemize}
\item $w_1^2=\rho w_1+\tau c_1$
\item $w_2^2=\rho^2 w_2 + \rho \tau \bigl (w_1c_1+w_1^{(1)}\bigr ) +
\tau^2 c_2$
\item $\bigl [w_1^{(1)}\bigr ]^2 = \rho\bigl (w_1^{(1)}c_1 +
w_1c_2\bigr )+\tau(c_1^3+c_1c_2)$
\item $w_1w_2=\rho w_2 + \tau\bigl (w_1c_1+w_1^{(1)}\bigr )$
\item $w_1w_1^{(1)}=\rho w_1^{(1)}+\tau c_1^2 +w_2c_1$
\item $w_2 w_1^{(1)}=\rho w_2 c_1 + \tau (w_1c_1^2+w_1^{(1)}c_1+w_1c_2)$.
\end{itemize}
The classes $1$, $w_1$, $w_2$, and $w_1^{(1)}$ give a free basis for
$H^{*,*}(\Gr_2(\cU);\Z/2)$ as a module over the subring $\M_2[c_1,c_2]$.

The forgetful map $H^{*,*}(\Gr_2(\cU);\Z/2) \ra
H^*(\Gr_2(\R^\infty);\Z/2)=\Z/2[w_1,w_2]$ from equivariant to
non-equivariant cohomology sends
\begin{itemize}
\item $\rho\mapsto 0$, $\tau\mapsto 1$
\item $w_1\mapsto w_1$, $w_2\mapsto w_2$
\item $c_1\mapsto w_1^2$, \quad $c_2\mapsto w_2^2$, \quad $w_1^{(1)}\mapsto w_1w_2+w_1^3$.  
\end{itemize}
(Note that the final line can be read off from the above relations and
the first two lines).  
\end{prop}

The complexity of the above description is discouraging, but the main
point is really that (a) it can be done, and (b) it is tedious but
mostly mechanical.   We discuss both the cases $k=2$ and $k=3$ in
detail in Section~\ref{se:examples}.

\begin{remark}
In nonequivariant topology there is the relation $c_i(E\tens
\C)=w_i^2(E)$. The first two relations in Proposition~\ref{pr:k=2}
should be thought of as deformations of this nonequivariant relation.
\end{remark}

One might expect the problem of describing the rings
$H^{*,*}(\Gr_k(\cU);\Z/2)$ to become
more tractable as $k\mapsto \infty$.   In some ways it does, but even
in this case we
have not found a convenient way to write down a complete set of relations.
See Proposition~\ref{pr:stable} for more information.

\subsection{Open questions}

\begin{enumerate}[(1)]
\item Our computations produce the full set of characteristic classes
for equivariant real vector bundles, taking values in
$H^{*,*}(\blank;\Z/2)$.  It remains to investigate possible uses for
such classes, and in particular their ties to geometry.

\item
In the classical case another way to describe the ring structure on the
cohomology of Grassmannians is combinatorially, via
Littlewood-Richardson rules.  It might be useful to work out
equivariant versions of these rules, and to describe the ring structure
that way instead of by generators and relations.  

\item There is an interesting duality that appears in our description
of  the cohomology ring
for $H^{*,*}(\Gr_k(\cU);\Z/2)$.  See Corollary~\ref{co:duality} and
the charts preceding it.  Is there
some geometry underlying this duality?

\item
Although we have computed the bigraded cohomology of the infinite
Grassmannians $\Gr_k(\cU)$, our techniques do not yield the cohomology
of the finite Grassmannians (in which $\cU$ is replaced by a
finite-dimensional subspace).  The reason is tied to our inability to
resolve all the differentials in the cellular spectral sequence.  
So computing the cohomology in these
cases remains an open problem.  

\item We have not developed any understanding of how to analyze
differentials in cellular spectral sequences, since the approach of
this paper essentially amounts to a sneaky way of avoiding this.  
Developing a method for computing such differentials, and connecting
them to geometry, is an important area for exploration.

\item If $\C^\infty$ is given the conjugation action, then the space
of complex $k$-planes $\Gr_k(\C^\infty)$ has simple cohomology, even
integrally: $H^{*,*}(\Gr_k(\C^\infty);\Z)=\Z[c_1,c_2,\ldots]$ where
the Chern classes $c_i$ have bidegree $(2i,i)$.  These are
the characteristic classes for Real vector bundles (where `Real' is in
the sense of Atiyah \cite{A2}).  One can attempt a similar computation
but replacing $\C^\infty$ with
$\C\tens \cU$: non-equivariantly this is still $\C^\infty$, but the
action is different---it is $\C$-linear rather than conjugate-linear.  The
computation of $H^{*,*}(\Gr_k(\C\tens \cU);\Z)$ seems to be an open
problem, that could perhaps be tackled by the methods of this paper.
See \cite{FL} for some relevant, early computations.

\item The initial motivation of this work was an interest in motivic
characteristic classes for quadratic bundles, generalizing the
Stiefel-Whitney classes of Delzant \cite{De} and Milnor \cite{M}; see
Section~\ref{se:motivic} for the connection with the present paper.  The original motivic
question remains unsolved.
\end{enumerate}

\subsection{Organization of the paper}
Section~\ref{se:background} gives some brief background about the theory
$H^{*,*}(\blank;\Z/2)$.  Section~\ref{se:add-basis} gives a first look at the ring of
invariants $[H^{*,*}(\Gr_1(\cU);\Z/2)^{\tens k}\Bigr ]^{\Sigma k}$,
and we provide an additive basis over the ground ring $\M_2$.  We also
measure the size of this ring by counting the elements of this free
basis that appear in each bidegree.

In Section~\ref{se:cells} we describe the equivariant Schubert-cell decomposition
of $\Gr_k(\cU)$.  A key point here is counting the number of Schubert
cells in each bidegree.  We also introduce the associated spectral
sequence for computing $H^{*,*}(\Gr_k(\cU);\Z/2)$, and in
Section~\ref{se:differentials} we discuss
Kronholm's theorems about this spectral sequence.

Section~\ref{se:proof} contains the main topological part of the paper.  Using the
results of Sections~\ref{se:add-basis}--\ref{se:differentials}
we prove that $H^{*,*}(\Gr_k(\cU);\Z/2)$ is isomorphic to the
expected ring of invariants (Theorem~\ref{th:main1}).  

In Section~\ref{se:mult} we turn to the multiplicative structure of our ring of
invariants.  We calculate some relations here, and we identify
a minimal set of generators.  This section is entirely algebraic.  
Section~\ref{se:examples} then gives a presentation for the ring of invariants in the
cases $k=2$ and $k=3$.

Finally, Section~\ref{se:motivic} describes the connection between the present work
and a certain motivic problem about characteristic classes of
quadratic bundles.  The results of this section are not needed
elsewhere in the paper.  An appendix is enclosed which calculates
the ring of invariants for $\Sigma_n$ acting on
$\Lambda_{\F_2}(a_1,\ldots,a_n)\tens_{\F_2} \F_2[b_1,\ldots,b_n]$
by permutation of the indices.  This purely algebraic result is needed
in the body of the text, and we were unable to find a suitable reference.

\vspace{0.1in}

Throughout this paper, if $X$ is a $\Z/2$-space then we write
$\sigma\colon X\ra X$ for the involution.  
For general background on $RO(G)$-graded equivariant cohomology
theories we refer the reader to \cite{Ma}.

\subsection{Acknowledgments} 
I am grateful to Mike Hopkins for a useful conversation about this
subject, and to John Greenlees for expressing some early interest.


\section{Background on equivariant cohomology}
\label{se:background}

Recall that $\M_2$ denotes the cohomology ring $H^{*,*}(\pt;\Z/2)$.
This ring is best depicted via the following diagram:

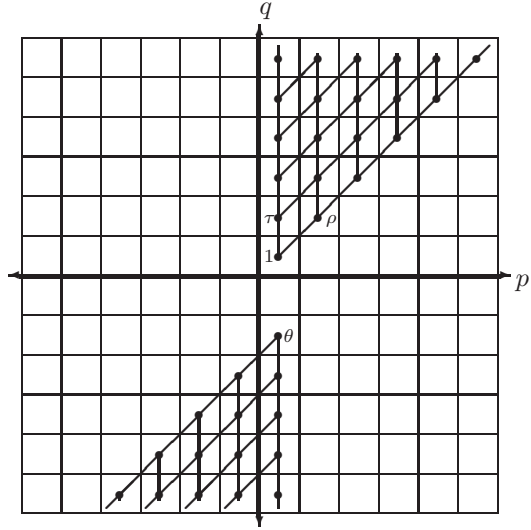
\begin{picture}(300,220)(-80,-30)
\multiput(-15,-15)(0,15){13}{\line(1,0){180}}
\multiput(-15,-15)(15,0){13}{\line(0,1){180}}
\multiput(82,82)(0,15){6}{\circle*{3}}
\multiput(97,97)(0,15){5}{\circle*{3}}
\multiput(112,112)(0,15){4}{\circle*{3}}
\multiput(127,127)(0,15){3}{\circle*{3}}
\multiput(142,142)(0,15){2}{\circle*{3}}
\multiput(157,157)(0,15){1}{\circle*{3}}
\multiput(82,52)(0,-15){5}{\circle*{3}}
\multiput(67,37)(0,-15){4}{\circle*{3}}
\multiput(52,22)(0,-15){3}{\circle*{3}}
\multiput(37,7)(0,-15){2}{\circle*{3}}
\multiput(22,-8)(0,-15){1}{\circle*{3}}
\put(-10,75){\vector(1,0){180}}
\put(-10,74){\line(1,0){176}}
\put(160,75){\vector(-1,0){180}}
\put(160,74){\line(-1,0){176}}
\put(75,-10){\vector(0,1){180}}
\put(75,20){\vector(0,-1){40}}
\put(74,-10){\line(0,1){176}}
\put(74,20){\line(0,-1){36}}
\thinlines
\put(172,71){$p$}
\put(75,174){$q$}
\put(100,95){$\scriptstyle{\rho}$}
\put(76.5,95){$\scriptstyle{\tau}$}
\put(76.5,80){$\scriptstyle{1}$}
\put(84,50){$\scriptstyle{\theta}$}
\thicklines
\put(82,97){\line(1,1){60}}
\put(82,112){\line(1,1){45}}
\put(82,127){\line(1,1){30}}
\put(82,142){\line(1,1){15}}
\put(97,97){\line(0,1){62}}
\put(112,112){\line(0,1){47}}
\put(127,127){\line(0,1){32}}
\put(142,142){\line(0,1){17}}
\put(82,82){\line(1,1){80}}
\put(82,82){\line(0,1){80}}
\put(82,52){\line(0,-1){65}}
\put(82,52){\line(-1,-1){65}}
\put(82,37){\line(-1,-1){50}}
\put(82,22){\line(-1,-1){35}}
\put(82,7){\line(-1,-1){20}}
\put(67,37){\line(0,-1){47}}
\put(52,22){\line(0,-1){32}}
\put(37,7){\line(0,-1){17}}
\put(22,-8){\line(0,-1){2}}
\end{picture}

Each dot represents a $\Z/2$, each vertical line represents a
multiplication by $\tau$, and each diagonal line represents
multiplication by $\rho$.  In the ``positive'' range $p,q \geq 0$, the
ring is therefore just $\Z/2[\tau,\rho]$.  In the negative range there
is an element $\theta\in \M_2^{0,-2}$ together with elements that one
can formally denote $\frac{\theta}{\tau^k\rho^l} \in
\M_2^{-l,-2-k-l}$.  After specifying $\theta^2=0$ 
this gives a complete description of the ring
$\M_2$.  We will refer to the subalgebra $\Z/2[\tau,\rho] \subseteq
\M_2$ as the \dfn{positive cone}, and the direct sum of all
$\M_2^{p,q}$ for $q<0$ will be called the \dfn{negative cone}.  
See \cite{C} and \cite{D} for more background on this coefficient ring.

There are natural transformations $H^{p,q}(X;\Z/2)\ra
H^p_{sing}(X;\Z/2)$ from our bigraded cohomology to ordinary singular
cohomology.  These are compatible with the ring structure, and when
$X$ is a point they send $\tau\mapsto 1$ and $\rho\mapsto 0$.  Since
everything in the negative cone is a multiple of $\rho$, it
follows that the entire negative cone of $\M_2$ is sent
to $0$.  

\subsection{The graded rank functor}
Let $I\inc \M_2$ denote the kernel of the projection map $\M_2 \ra
\Z/2$.  Let $M$ be a bigraded, finitely-generated free module over $\M_2$.
Define the \dfn{bigraded rank} of $M$ by the formula
\[ \rank^{p,q} M = \dim_{\Z/2} (M/IM)^{p,q}.
\]
So $\rank M$ should be regarded as a function 
$\Z^2 \ra \Z_{\geq 0}$.
Clearly $M$ is determined, up to isomorphism, by its bigraded rank.

It is usually easiest to depict the bigraded rank as a chart.  For
example, the bigraded rank of $H^{*,*}(\Gr_1(\cU);\Z/2)$ is

\begin{picture}(300,120)(-50,0)
\multiput(0,0)(0,15){8}{\line(1,0){165}}
\multiput(0,0)(15,0){12}{\line(0,1){105}}
\put(5,5){$1$}
\put(20,20){$1$}
\put(35,20){$1$}
\put(50,35){$1$}
\put(65,35){$1$}
\put(80,50){$1$}
\put(95,50){$1$}
\put(110,65){$1$}
\put(125,65){$1$}
\put(140,80){$1$}
\put(155,80){$1$}
\put(168,-2){$p$}
\put(-7,102){$q$}
\put(0,0){\circle*{4}}
\end{picture}

\vspace{0.2in}

\noindent
where the lower left corner is the $(0,0)$ spot and all unmarked boxes
are regarded as having a $0$ in them.

\section{An additive basis for  the ring of invariants}
\label{se:add-basis}

Let $R=\M_2[a,b]/(a^2=\rho a+\tau b)$ where $a$ has degree $(1,1)$ and
$b$ has degree $(2,1)$.  Fix $k\geq 1$ and let $T_k=R^{\tens k}$.  Let
$\Sigma_k$ act on $T_k$ in the evident way, as permutation of the
tensor factors.  Define
$\Inv_k = [T_k]^{\Sigma_k}$.
Our goal in this section is to investigate an additive basis for the
algebra $\Inv_k$, regarded as a module over $\M_2$.  The
multiplicative structure of this ring will be discussed in
Section~\ref{se:mult}.    

\medskip

It will be convenient to rename the variables in the $i$th copy of $R$
as $a_i$ and $b_i$.  So $T_k$ is the quotient of
$\M_2[a_1,\ldots,a_k,b_1,\ldots,b_k]$
by the relations $a_i^2=\rho a_i+\tau b_i$, for $1\leq i\leq k$.  Let 
\[ w_i=\sigma_i(a_1,\ldots,a_k) \]
be the $i$th elementary symmetric function in the $a$'s, and
let
\[ c_i=\sigma_i(b_1,\ldots,b_k)
\]
be the $i$th elementary symmetric function in the $b$'s.  These are
the most obvious elements of $\Inv_k$, but there are others as well.
For example, the element $a_1b_1+a_2b_2+\cdots+a_kb_k$ is invariant
under the action of $\Sigma_k$.  We will need some notation to help us
describe these other elements of $\Inv_k$. 

If $m$ is a
monomial in the $a$'s and $b$'s, write $[m]$ for the smallest
homogeneous polynomial in $T_k$ which contains $m$ as one of its terms.
By `smallest' we mean the smallest number of monomial summands.  
If $H\leq \Sigma_k$ is the stabilizer of $m$, then $[m]$ is the sum
$\sum_{gH\in \Sigma_k/H} gm$.  
Here are some examples:
\begin{enumerate}[(i)]
\item
$[a_1b_1]=a_1b_1+a_2b_2+\ldots+a_kb_k$
\item 
$[a_1b_1b_2]=\sum\limits_{i \neq j} a_ib_ib_j$
\item
$[a_1b_2b_3]=\sum\limits_{\text{$i,j,k$ distinct}}a_ib_jb_k$
\item $[a_1a_2]=w_2$.
\end{enumerate}

Notice that 
\[ [a_1^2b_2]=\sum_{i\neq j} a_i^2b_j = \sum_{i\neq j} (\rho
a_i +\tau b_i)b_j= \rho \sum_{i\neq j} a_ib_j + \tau \sum_{i\neq
j} b_ib_j = \rho[a_1b_2]+\tau[b_1b_2].
\]
A similar computation shows that if $m$ is any monomial with an
$a_i^2$ then $[m]$ is an $\M_2$-linear combination of monomials
$[m_j]$ with $\deg m_j < \deg m$.  

The following proposition is fairly clear:

\begin{prop}
As an $\M_2$-module, $\Inv_k$ is free with basis consisting of all elements
$[a_1^{\epsilon_1}\ldots a_k^{\epsilon_k}b_1^{d_1}\ldots b_k^{d_k}]$,
where each $d_i \geq 0$ and each $\epsilon_i \in \{0,1\}$.  
\end{prop}

Our next task is to count how many of the above basis elements appear
in any given bidegree.  Write $\Mon_k$ for the set of monomials in the
variables $a_1,\ldots,a_k,b_1,\ldots,b_k$ having the property that the
exponent on each $a_i$ is at most $1$.  The above proposition implies
that $\Inv_k$ has a free basis over $\M_2$ that is in bijective
correspondence with the set of orbits $\Mon_k\!/\Sigma_k$; this
correspondence preserves the bigraded degree.
From now on we will refer to this basis 
as THE free basis for $\Inv_k$.  We can easily write down a list
of these basis elements in any given bidegree.  For instance, here is
the list in low dimensions, assuming $k$ is large (with the degrees of
the elements given to the left):
\[ \xymatrixrowsep{0.2pc}\xymatrix{
(1,1): & [a_1] && (4,3): & [a_1a_2b_1], [a_1a_2b_3] \\
(2,1): & [b_1] && (4,4): & [a_1a_2a_3a_4] \\
(2,2): & [a_1a_2] && (5,3): & [a_1b_1^2], [a_1b_2^2], [a_1b_1b_2], [a_1b_2b_3] \\
(3,2): & [a_1b_1], [a_1b_2] && (5,4): & [a_1a_2a_3b_1], [a_1a_2a_3b_4]  \\
(3,3): & [a_1a_2a_3] && (5,5): &  [a_1a_2a_3a_4a_5]\\
(4,2): & [b_1^2], [b_1b_2] && (6,3): & [b_1b_2b_3],[b_1^3],[b_1^2b_2] \\
}
\]

We can count the number of generators in each bidegree in terms of
certain kinds of partitions.  Given $n$, $k$, and $j$, let
$\prt{n}{k}{j}$ denote the number of partitions of $n$ into $k$
nonnegative integers such that exactly $j$ of the integers are odd.
For example, $\prt{8}{5}{4}= 4$ because it counts the following
partitions: $01133$, $01115$, $11114$, and $11123$.

\begin{prop}
\label{pr:rank-Inv}
For any $p$, $q$, and $k$ one has
$\rank^{p,q}(\Inv_k)=\prt{p}{k}{2q-p}$.
\end{prop} 

\begin{proof}
Let $w$ be a monomial in the variables
$a_1,\ldots,a_k,b_1,\ldots,b_k$ where each $a_i$ appears at most once.
We'll say that $w$ is \dfn{pure} if
all the symbols in $w$ have the same subscript: e.g., $a_1b_1^3$ is
pure, but $a_1a_2b_1^2$ is not.  The monomial $w$ can be written in a
unique way as $w=w(1)w(2)\cdots w(k)$ where each $w(i)$ is pure and
only contains the subscript $i$.  

Regard $a_i$ as having degree $1$ and $b_i$ as having degree $2$.
If $v$ is a pure monomial, let $d(v)$ be its total degree.
Finally, if $w$ is any monomial then let $\eta(w)$ be the partition
\[ \eta(w)=[d(w(1)),\,d(w(2)),\,\ldots,\,d(w(k))].
\]
For example, if $w=a_1a_2a_3b_1^2b_2b_4$ then $\eta(w)=[5312]$.  

It is clear that all the $\Sigma_k$-cognates of $w$ give rise to the
same partition, and so we have a function
\[ \Mon_k\!/\Sigma_k \llra{\eta} \{\text{partitions with $\leq k$
pieces}\}.
\]  
Moreover, this is a bijection because the partition is enough to
recover the invariant element $[w]$: if the $i$th number in our
partition is $2r$ then we write $b_i^r$, and if it is $2r+1$ we write
$a_ib_i^r$, and then we multiply these terms together.  For example,
given the partition $[34678]$ we would write
$[a_1b_1b_2^2b_3^3a_4b_4^3b_5^4]$.  This apparently depends on the
order in which we listed the numbers in the partition, but this
dependence goes away when we take the $\Sigma_k$-orbit.

Clearly the topological degree of the monomial $w$ equals the sum of
the elements in the partition $\eta(w)$.  Also, the number of odd
elements of the partition is equal to the number of $a_i$'s in $w$.
But one readily checks that 
\begin{align*} \text{weight of $w$} = \#b_i\text{'s} + \# a_i\text{'s} &=
\frac{\text{(topl. degree of $w$)}-\#a_i\text{'s}}{2} \, \,+ \, \# a_i\text{'s}\\
& =
\frac{\text{(topl. degree of $w$)}
+ \#a_i\text{'s}}{2}.
\end{align*}
So the number of odd elements in the partition $\eta(w)$ is $2q-p$,
where $q$ is the weight of $w$ and $p$ is the topological degree of
$w$.
\end{proof}

As an example of the above proposition, here is a portion of the
bigraded rank function for $\Inv_4$:

\begin{picture}(300,140)(-50,0)
\multiput(0,0)(0,15){10}{\line(1,0){225}}
\multiput(0,0)(15,0){16}{\line(0,1){135}}
\multiput(5,5)(15,15){5}{$1$}
\put(35,20){$1$}
\multiput(50,35)(15,15){3}{$2$}
\put(80,50){$4$}
\put(65,35){$2$}
\put(95,50){$3$}
\put(95,65){$5$}
\put(95,80){$1$}
\put(110,65){$7$}
\put(110,80){$4$}
\put(125,65){$5$}
\put(125,80){$8$}
\put(125,95){$2$}
\put(138,80){$11$}
\put(140,95){$7$}
\put(155,80){$6$}
\put(152,95){$14$}
\put(155,110){$3$}
\put(168,95){$16$}
\put(168,110){$11$}
\put(185,95){$9$}
\put(182,110){$20$}
\put(185,125){$5$}
\put(198,110){$23$}
\put(198,125){$16$}
\put(213,110){$11$}
\put(213,125){$30$}
\put(228,-2){$p$}
\put(-7,132){$q$}
\put(0,0){\circle*{4}}
\end{picture}

And here is a similar chart for $\Inv_5$:

\begin{picture}(300,150)(-50,-5)
\multiput(0,0)(0,10){15}{\line(1,0){220}}
\multiput(0,0)(10,0){23}{\line(0,1){140}}
\multiput(2.5,2.5)(10,10){6}{$\scriptstyle{1}$}
\put(22.5,12.5){$\scriptstyle{1}$}
\multiput(32.5,22.5)(10,10){4}{$\scriptstyle{2}$}
\put(72.5,62.5){$\scriptstyle{1}$}
\multiput(42.5,22.5)(50,50){2}{$\scriptstyle{2}$}
\multiput(52.5,32.5)(30,30){2}{$\scriptstyle{4}$}
\multiput(62.5,42.5)(10,10){2}{$\scriptstyle{5}$}
\multiput(62.5,32.5)(50,50){2}{$\scriptstyle{3}$}
\multiput(72.5,42.5)(30,30){2}{$\scriptstyle{7}$}
\multiput(82.5,52.5)(10,10){2}{$\scriptstyle{9}$}
\multiput(82.5,42.5)(50,50){2}{$\scriptstyle{5}$}
\multiput(91.5,52)(30,30){2}{$\scriptstyle{12}$}
\multiput(101,62)(10,10){2}{$\scriptstyle{16}$}
\multiput(102.5,52.5)(50,50){2}{$\scriptstyle{7}$}
\multiput(111.5,62)(30,30){2}{$\scriptstyle{18}$}
\multiput(122,72)(10,10){2}{$\scriptstyle{25}$}
\multiput(121.5,62.5)(50,50){2}{$\scriptstyle{10}$}
\multiput(131.5,72.5)(30,30){2}{$\scriptstyle{27}$}
\multiput(141,82.5)(10,10){2}{$\scriptstyle{39}$}
\multiput(141.5,72.5)(50,50){2}{$\scriptstyle{13}$}
\multiput(151.5,82)(30,30){2}{$\scriptstyle{38}$}
\multiput(161,92)(10,10){2}{$\scriptstyle{56}$}
\multiput(161.5,82.5)(50,50){2}{$\scriptstyle{18}$}
\multiput(171.5,92)(30,30){2}{$\scriptstyle{53}$}
\multiput(181,102)(10,10){2}{$\scriptstyle{80}$}
\multiput(181.5,92.5)(50,50){1}{$\scriptstyle{23}$}
\multiput(191.5,102)(30,30){1}{$\scriptstyle{71}$}
\multiput(200,112)(10,10){2}{$\scriptscriptstyle{109}$}
\multiput(201.5,102.5)(50,50){1}{$\scriptstyle{30}$}
\multiput(211.5,112)(30,30){1}{$\scriptstyle{94}$}
\put(225,-2){$p$}
\put(-7,135){$q$}
\put(0,0){\circle*{4}}
\end{picture}

There are some evident patterns in these charts.  For example,
if one starts at spot $(2p,p)$ and reads diagonally upwards along a line
of slope $1$ then the resulting numbers have an evident symmetry.
This comes from a symmetry of the $\prt{n}{k}{j}$ numbers:

\begin{lemma}
\label{le:part-duality}
For any $n$, $k$, and $j$, one has
$\prt{n}{k}{j}=\prt{n+(k-2j)}{k}{k-j}$.  
\end{lemma}

\begin{proof}
Suppose $u_1,\ldots,u_k$ is a partition of $n$ in which there are
exactly $j$ odd numbers---we can arrange the indices so that these are
$u_1,\ldots,u_j$.  Subtract $1$ from all the odd numbers and add $1$
to all the even numbers: this yields the collection of numbers
$u_1-1,\ldots,u_j-1,u_{j+1}+1,\ldots,u_k+1$.  This is a partition of
$n+k-2j$ in which there are exactly $k-j$ odd numbers.  One readily
checks that this gives a bijection between the two kinds of partitions.
\end{proof}

The diagonal symmetries in our rank charts are as follows:

\begin{cor}
\label{co:duality}
For any $p$, $r$, and $k$, one has
\[ \rank^{2p+r,p+r}(\Inv_k)=\rank^{2p+k-r,p+k-r}(\Inv_k).
\]
\end{cor}

\begin{proof}
This is immediate from Proposition~\ref{pr:rank-Inv} and Lemma~\ref{le:part-duality}.
\end{proof}

The numbers in the rank chart for $\Inv_k$ organize themselves
naturally into lines of slope $\frac{1}{2}$.  To explain this (and
because it will be needed later) we introduce the following
terminology.  A \dfn{successor} of a partition $\alpha$ is any
partition obtained by adding $2$ to exactly one of the numbers in $\alpha$.
For example, $011$ has two successors: $013$ and $112$.  If a
partition $\beta$ is obtained from $\alpha$ by a sequence of
successors, we say that $\beta$ is a \dfn{descendent} of $\alpha$.
Finally, a partition $\alpha$ will be called \dfn{minimal} if it is
not a successor of any other partition.

For the set of all partitions consisting of $k$ nonnegative numbers, the
following facts are immediate:
\begin{enumerate}[(1)]
\item There are exactly $k+1$ minimal partitions: $00\ldots 0$, $00\ldots 01$,
$00\ldots 011$, 
$\ldots$, and $11\ldots 1$.  
\item Every partition $\alpha$ is a descendent of a unique minimal
partition, namely the one obtained by replacing each $\alpha_i$ with
either $0$ or $1$ depending on whether  $\alpha_i$ is even or odd.
\end{enumerate}

The partitions consisting of $k$ nonnegative numbers, with exactly $j$
odd numbers, form a tree under the successor operation: and the numbers
of such partitions forms 
the line of slope $\frac{1}{2}$ ascending from spot $(j,j)$ in our
rank charts.  

The following corollary records the evident bounds on the nonzero
numbers in our rank charts.  The proof is immediate from the things we
have already said, or it could be proven directly from Proposition~\ref{pr:rank-Inv}.

\begin{cor}
\label{co:bounds}
The bigraded rank function of $\Inv_k$ is nonzero only in the region
bounded by the three lines $y=x$, $y=\frac{1}{2}x$, 
and $y=\frac{1}{2}x+\frac{k}{2}$.  That is to
say, the elements of our free basis for $\Inv_k$ appear only in bidegrees
$(a,b)$ where $\frac{a}{2}\leq b\leq a$ if $a\leq k$, and $\frac{a}{2}
\leq b \leq \frac{1}{2}a+\frac{k}{2}$ if $a\geq k$.  
\end{cor}


\section{Schubert cells and a spectral sequence}
\label{se:cells}

Given a sequence of integers $1\leq a_1 < a_2 <
\cdots < a_k$, define the associated Schubert cell in $\Gr_k(\cU)$ by
\[ \Omega_a = \{ V \in \Gr_k(\cU)\,|\, \dim(V\cap \cU^{a_i}) \geq i \}.
\]
Here $\cU^{n}\subseteq \cU$ is simply the subspace of vectors whose
$r$th coordinates all vanish for $r>n$, which we note is closed under
the $\Z/2$-action.    It will be convenient for us to regard the
$a$-sequence as giving a ``$*$-pattern'', in which one takes an
infinite sequence of empty boxes and places a single $*$ in each box
corresponding to an $a_i$.  If the boxes represent the standard basis
elements of $\cU$, then the $*$'s represent where the jumps in
dimension occur for subspaces $V$ lying in the interior of $\Omega_a$.
These $*$-patterns will be used several times in our discussion below.
 
It is somewhat more typical to use a different indexing convention
here.  Define $\sigma_i=a_i-i$, so that we have
$0\leq \sigma_1 \leq \sigma_2\leq \cdots \leq \sigma_k$.  Write
$\Omega(\sigma)$ for the same Schubert cell as $\Omega_a$, which has dimension equal
to $\sum_i \sigma_i$.  Define a \mdfn{$k$-Schubert symbol} to be
an increasing sequence $\sigma_1\leq \sigma_2\leq \cdots \leq \sigma_k$.
To get the associated $*$-pattern, skip over $\sigma_1$ empty boxes
and then place a $*$; then skip over $\sigma_2-\sigma_1$ empty boxes
and place another $*$; then skip over $\sigma_3-\sigma_2$ empty boxes,
and so forth.  For example, the Schubert symbol $[0235]$ corresponds
to the $*$-pattern $[* \bblank \bblank * \bblank * \bblank \bblank *]$, or
the $a$-sequence $(1,4,6,9)$.

Let $F_r\subseteq \Gr_k(\cU)$ be the union of all the Schubert cells
of dimension less than or equal to $r$.  This filtration gives rise to
a spectral sequence on cohomology in the usual way, where the
$E_1$-term is the direct sum $\oplus_{\sigma}
\tilde{H}^{*,*}(S^{a_\sigma,b_\sigma})$ where $\sigma$ ranges over all
$k$-Schubert symbols and $(a_\sigma,b_\sigma)$ is the bidegree of the
associated cell.  We will next describe an algorithm for producing
this bidegree.

Picture
the row of symbols $+-+-+-\cdots$ going on forever, with the initial
symbol regarded as the first (rather than the zeroth).  These symbols
represent the $\Z/2$-action on the standard basis elements of $\cU$.  For each 
$i$ in the range $1\leq i\leq k$, change the $a_i$th symbol to an
asterisk $*$.  Then for each $i$, define $u_i$ to be
\[ u_i=\begin{cases} 
\text{the total number of $+$ signs to the left
of the $i$th asterisk} & \text{if $a_i$ is even} \\
\text{the total number of $-$ signs to the left
of the $i$th asterisk} & \text{if $a_i$ is odd.}
\end{cases}
\]
Finally, define the \dfn{cell-weight} of the Schubert symbol to be
$\sum_i u_i$.  We claim that the open Schubert cell corresponding to
$\sigma$ is isomorphic to $\R^{n,k}$ where $n=\sum \sigma_i$ and $k$
is the cell-weight of $\sigma$.

Let us say the above in a slightly different way.  We think in terms
of $*$-patterns, but where the boxes contain alternating $+$ and $-$
signs and the $*$'s eradicate whatever sign was in their box.  For the
topological dimension of a cell, we count the number of empty boxes to
the left of each $*$ and add these up.  For the weight we do a fancier
kind of counting: if the $*$ replaced a $+$ sign then we count the
number of $-$ signs to the left of it, whereas if it replaced a $-$
we count the number of $+$ signs to the left.  And again, we add
up our answers for each $*$ in the pattern to get the total weight.  
For example, consider the Schubert symbol $\sigma=[135]$ which has
topological dimension $9$.  The
corresponding $a$-sequence is $(2,5,8)$, and this gives the $*$-pattern
$+*+-*-+*+-+-\cdots$  So $u_1=1$, $u_2=1$, $u_3=3$, and therefore the
bidegree of $\Omega(\sigma)$ is $(9,5)$.

\begin{ex}
Consider the Grassmannian $\Gr_2(\cU^6)$.  There are
$\tbinom{6}{2}=15$ Schubert cells.  We list all the $*$-patterns and
the bidegrees of the associated cells:
\begin{align*}
&**+-+- \quad (0,0)  \qquad +**-+- \quad (2,1) \qquad +-*-*- \quad (5,3) \\
&*-*-+- \quad (1,1) \qquad +*+*+- \quad (3,3) \qquad +-*-+* \quad (6,3) \\
&*-+*+- \quad (2,1) \qquad +*+-*- \quad (4,2) \qquad +-+**- \quad (6,3) \\
&*-+-*- \quad (3,2) \qquad +*+-+* \quad (5,4) \qquad +-+*+* \quad (7,5) \\
&*-+-+* \quad (4,2) \qquad +-**+- \quad (4,2) \qquad +-+-** \quad (8,4) \\
\end{align*}
\end{ex}  

We need to justify our procedure for determining the weight of a
Schubert cell.  Given an $a$-sequence, points in the interior of the
associated Schubert cell $\Omega_a$ are in bijective correspondence
with 
matrices of a form such as
\[ \begin{bmatrix}
? & ? & 1 & 0 & 0 & 0 & 0 & 0\\
? & ? & 0 & ? & 1 & 0 & 0 & 0 \\
? & ? & 0 & ? & 0 & ? & ? & 1 
\end{bmatrix}.
\]
(The matrix given is for the case of $\Gr_3(\cU)$ and the $a$-sequence
$(3,5,8)$).  The matrix in question has $1$'s in the columns given by
the $a$-sequence, each $1$ is followed by only zeros in its row, and
each $1$ is the only nonzero entry in its column.  The set of such
matrices is a Euclidean space of dimension equal to the number of
``?'' symbols.  Such a matrix determines a point in $\Gr_k(\cU)$ by
taking the span of its rows, and any $k$-plane in the interior of
$\Omega_a$ has a unique basis of the above form.  This is all standard
from non-equivariant Schubert calculus.  In the equivariant case, we
have a $\Z/2$-action on the set of such matrices induced by the
$\Z/2$-action on $\cU$.  In our above example, the action is
\[
\begin{bmatrix}
b & c & 1 & 0 & 0 & 0 & 0 & 0\\
d & e & 0 & f & 1 & 0 & 0 & 0 \\
g & h & 0 & i & 0 & j & k & 1 
\end{bmatrix}\mapsto
\begin{bmatrix}
b & -c & 1 & 0 & 0 & 0 & 0 & 0\\
d & -e & 0 & -f & 1 & 0 & 0 & 0 \\
g & -h & 0 & -i & 0 & -j & k & -1 
\end{bmatrix}.
\]
Notice that the matrix on the right is not in our standard form.  To
convert it to standard form we multiply the third row by $-1$ to get
\[
\begin{bmatrix}
b & c & 1 & 0 & 0 & 0 & 0 & 0\\
d & e & 0 & f & 1 & 0 & 0 & 0 \\
g & h & 0 & i & 0 & j & k & 1 
\end{bmatrix}\mapsto
\begin{bmatrix}
b & -c & 1 & 0 & 0 & 0 & 0 & 0\\
d & -e & 0 & -f & 1 & 0 & 0 & 0 \\
-g & h & 0 & i & 0 & j & -k & 1 
\end{bmatrix}.
\]
So as a $\Z/2$-representation we have $\R^{10}$ with five sign changes,
and this is $\R^{10,5}$.  It is now easy to go from this overall
picture to the specific formula for the cell-weight that was given
above.

\vspace{0.1in}

We now know how to compute the bigraded Schubert cell decomposition
for any Grassmannian.  It is useful to look at a specific example, so
here is the Schubert cell picture for $\Gr_5(\cU)$.  Each box gives
the number of Schubert cells of the given bidegree.  

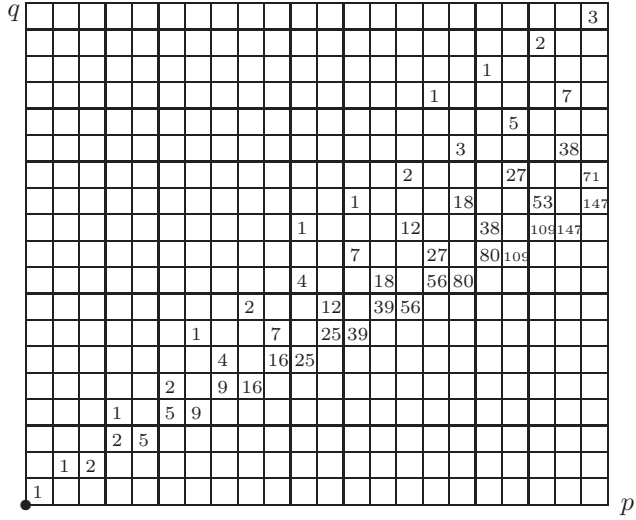
\begin{picture}(300,200)(-50,0)
\multiput(0,0)(0,10){20}{\line(1,0){220}}
\multiput(0,0)(10,0){23}{\line(0,1){190}}

\put(2.5,2.5){$\scriptstyle{1}$}
\put(22.5,12.5){$\scriptstyle{2}$}
\put(42.5,22.5){$\scriptstyle{5}$}
\put(62.5,32.5){$\scriptstyle{9}$}
\put(81.5,42.5){$\scriptstyle{16}$}
\put(101.5,52.5){$\scriptstyle{25}$}
\put(121.5,62.5){$\scriptstyle{39}$}
\put(141.5,72.5){$\scriptstyle{56}$}
\put(161.5,82.5){$\scriptstyle{80}$}
\put(180.5,92.5){$\scriptscriptstyle{109}$}
\put(200.5,102.5){$\scriptscriptstyle{147}$}
\put(12.5,12.5){$\scriptstyle{1}$}
\put(32.5,22.5){$\scriptstyle{2}$}
\put(52.5,32.5){$\scriptstyle{5}$}
\put(72.5,42.5){$\scriptstyle{9}$}
\put(91.5,52.5){$\scriptstyle{16}$}
\put(111.5,62.5){$\scriptstyle{25}$}
\put(131.5,72.5){$\scriptstyle{39}$}
\put(151.5,82.5){$\scriptstyle{56}$}
\put(171.5,92.5){$\scriptstyle{80}$}
\put(190.5,102.5){$\scriptscriptstyle{109}$}
\put(210.5,112.5){$\scriptscriptstyle{147}$}
\put(32.5,32.5){$\scriptstyle{1}$}
\put(52.5,42.5){$\scriptstyle{2}$}
\put(72.5,52.5){$\scriptstyle{4}$}
\put(92.5,62.5){$\scriptstyle{7}$}
\put(111.5,72.5){$\scriptstyle{12}$}
\put(131.5,82.5){$\scriptstyle{18}$}
\put(151.5,92.5){$\scriptstyle{27}$}
\put(171.5,102.5){$\scriptstyle{38}$}
\put(191.5,112.5){$\scriptstyle{53}$}
\put(210.5,122.5){$\scriptscriptstyle{71}$}
\put(62.5,62.5){$\scriptstyle{1}$}
\put(82.5,72.5){$\scriptstyle{2}$}
\put(102.5,82.5){$\scriptstyle{4}$}
\put(122.5,92.5){$\scriptstyle{7}$}
\put(141.5,102.5){$\scriptstyle{12}$}
\put(161.5,112.5){$\scriptstyle{18}$}
\put(181.5,122.5){$\scriptstyle{27}$}
\put(201.5,132.5){$\scriptstyle{38}$}
\put(102.5,102.5){$\scriptstyle{1}$}
\put(122.5,112.5){$\scriptstyle{1}$}
\put(142.5,122.5){$\scriptstyle{2}$}
\put(162.5,132.5){$\scriptstyle{3}$}
\put(182.5,142.5){$\scriptstyle{5}$}
\put(202.5,152.5){$\scriptstyle{7}$}
\put(152.5,152.5){$\scriptstyle{1}$}
\put(172.5,162.5){$\scriptstyle{1}$}
\put(192.5,172.5){$\scriptstyle{2}$}
\put(212.5,182.5){$\scriptstyle{3}$}
\put(225,-2){$p$}
\put(-7,185){$q$}
\put(0,0){\circle*{4}}
\end{picture} 

\vspace{0.1in}

Note that the numbers appearing along lines of slope $\frac{1}{2}$
are the same as the numbers we saw in the rank chart for $\Inv_5$,
except that the lines are arranged differently in the plane.  We will need
a precise statement:

\begin{prop}
\label{pr:cell-bound}
Let $X$ be the $E_1$-term of the cellular spectral sequence for
$\Gr_k(\cU)$ based on the Schubert cell filtration.    Then the
nonzero entries in the rank chart for $X$ are bordered by the lines
$y=x$, $y=\frac{1}{2}x$, and $y=\frac{1}{2} (x+\binom{k+1}{2})$.  

Moreover, for any $j,r$ one has $\rank^{j+2r,j+r}X=0$ unless
$j=\binom{i}{2}$ for some $i$ in the range $1\leq i \leq k+1$.
And finally, if $j=\binom{i}{2}$ then
\[ \rank^{j+2r,j+r}X = \partition_{2r+\gamma_i,\leq k}[\gamma_i]=
\rank^{\gamma_i+2r,\gamma_i+r}(\Inv_k)
\]
where $\gamma$ is the function defined by
\[ \gamma_i=\begin{cases}
\tfrac{k+i}{2} & \text{if $k-i$ is even} \\
\tfrac{k+1-i}{2} & \text{if $k-i$ is odd}.
\end{cases}
\]
\end{prop}

The mathematical phrasing of the above proposition is somewhat awkward, but it
says something very concrete.  Namely, the nonzero entries in the rank
chart for $X$ 
are divided into rays of slope $\frac{1}{2}$ emanating from the points
$\bigl (\binom{i}{2},\binom{i}{2}\bigr )$ for $1\leq i \leq k+1$.  
Starting from $\bigl (\binom{k+1}{2},\binom{k+1}{2}\bigr )$ and working towards
the origin along the $y=x$ line, mentally label each vertex with the numbers in
the sequence
\[ 0,\ k,\ 1,\ k-1, \ 2,\ k-2,\ 3,\ k-3,\ \ldots \]
These are the numbers $\gamma_{k+1}, \gamma_{k},\ldots$.   Then in $\rank(X)$, the
 $r$th term from $\bigl (\binom{i}{2},\binom{i}{2}\bigr )$ along the ray of
slope $\frac{1}{2}$ is equal to  $\partition_{2r+\gamma_i,\leq
k}[\gamma_i]$.

In order to prove Proposition~\ref{pr:cell-bound} we need to introduce
some language for
bookkeeping.  
Define a \dfn{successor} of a $*$-pattern to be a pattern made by
moving one of the $*$'s two spots to the right (note that one can only
do this if the new spot for the $*$ started out empty).  In terms of
$a$-sequences, a successor is an $a$-sequence obtained by adding $2$
to one of the $a_i$'s.  For example, the $a$-sequence $123$ has
exactly two successors, namely $125$ and $145$.  A $*$-pattern (or
$a$-sequence) is said to be
\dfn{minimal} if it is not the successor of another pattern (or
sequence); said differently, a $*$-pattern is minimal if one cannot
move any $*$ two places to the left.    
The sequences $123$ and $124$ are both minimal, but $125$ is not; these correspond to the
$*$-patterns $[***]$, $[**+*]$, and $[**+-*]$.  

Observe that taking the successors of a $*$-pattern increases the
bidegree of the associated Schubert cell by $(2,1)$.  This is easy to
explain in terms of the following picture, showing an arbitrary $*$-pattern and a
successor obtained by moving one of the $*$'s:  

\vspace{0.1in}

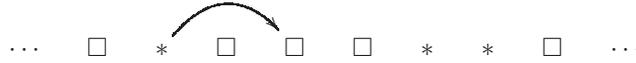
\[ \xymatrixcolsep{1pc}\xymatrix{
\cdots & \square & {*}  \ar@/^4ex/[rr]& \square & \square & \square & {*}
& {*} & \square & \cdots
}
\]

\vspace{0.1in}

\noindent
The count of empty boxes to the left of each $*$ is the same for the
two patterns, except for the $*$ that got moved: and for that $*$ the
count has increased by $2$.  Likewise, the number of $+/-$ signs in
the empty boxes stays the same for each $*$ in the two patterns,
except again for the $*$ that got moved: and for that $*$ the number
of $+$ and $-$ signs to the left of it each got increased by $1$.  

The fact that the successor relation increases the bidegree by $(2,1)$
explains why our Schubert cell chart breaks up into rays of slope
$\frac{1}{2}$.  The number of such rays will be governed by the number
of minimal $*$-patterns, so we investigate this next.

It is clear that for a $*$-pattern to be minimal it must be true that any two
successive $*$'s have at most one empty space between them.  Moreover,
as soon as one has an empty space in the $*$-pattern then all
successive $*$'s must be separated by one empty space.  So for
patterns with $k$ asterisks, there are exactly $k+1$ minimal patterns;
they are completely described by saying which $*$ has the first blank
space after it (the count is $k+1$ because the first blank might
appear after the {\it zeroth\/} star, which doesn't actually exist). 
One thing that is easy to verify about these minimal patterns is that
the corresponding Schubert cells each have bidegree $(p,p)$, for some
values of $p$; that is, the topological dimension and weight coincide.  
Recall that computing both the topological dimension and the weight
from the $*$-patterns amounts to counting empty boxes to the left of
each $*$, with the weight computation involving some restrictions on
which boxes get counted.  For the minimal $*$-patterns, the placement
of the $*$'s results in these restrictions all being vacuous: that is,
all empty boxes are counted.  

The minimal $*$-patterns correspond to the following $a$-sequences:
\[ (1,2,3,\ldots,k),(1,2,\ldots,k-1,k+1),(1,2,\ldots,k-2,k,k+2),\ldots,(1,3,5,\ldots,2k-1)
\]
and $(2,4,6,\ldots,2k)$
(the first $k$ of these follow a common pattern, the final one does not).
The associated Schubert symbols are
\[ [00\ldots 0], \ [00\ldots 01], \ [00\ldots 12],\ \ldots,\  [012\ldots (k-1)],\quad\text{and}\quad
[123\ldots k].
\]
The topological dimensions are therefore $\binom{i}{2}$ for $1\leq
i\leq k+1$, so the minimal $*$-patterns correspond to
Schubert cells of bidegree $(\binom{i}{2},\binom{i}{2})$ for $i$ in
this range.

\begin{proof}[Proof of Proposition~\ref{pr:cell-bound}]
We have determined in the preceding discussion that the successor
relation breaks the Schubert-cell chart into $k+1$ rays of slope
$\frac{1}{2}$, each ray starting at a point
$(\binom{i}{2},\binom{i}{2})$ for $1\leq i\leq k+1$.  The starting
points are the minimal $*$-patterns determined above.  What remains to
be shown is that the number of cells counted along these rays matches
similar rays in the count of partitions we saw in our study of
$\rank(\Inv_k)$.  This is where the awkward rearrangement of the rays
must be accounted for.

We have the classical bijection between Schubert cells and partitions,
which associates to any $*$-pattern the corresponding Schubert symbol.
For the rest of this proof we completely discard this bijection, and
instead use a {\it different\/} bijection, to be described next.  This
is the crux of the argument.  See Remark~\ref{re:strange-bijection}
below for more information about where this new bijection comes from.

Given a partition $\sigma$ with $k$ nonnegative parts, regard this
as two partitions $\sigma^{ev}$ and $\sigma^{odd}$ by simply
separating the even and odd numbers.  For example, if $\sigma=[00123]$
then $\sigma^{ev}=[002]$ and $\sigma^{odd}=[13]$.  
Note that in both $\sigma^{ev}$ and $\sigma^{odd}$ the difference of
consecutive pieces (when ordered from least to greatest) will always
be even.  

Consider a string of empty boxes labelled $1,2,3,\ldots$  
Take $\sigma^{ev}$ and convert this
to a $*$-pattern in what is essentially the usual way, but placing the
$*$'s only in the {\it even\/} boxes of the pattern.  
If
$\sigma^{ev}=[u_1,\ldots,u_r]$ then skip over $\frac{u_1}{2}$ even
boxes and place a $*$, then skip over $\frac{u_2-u_1}{2}$ even boxes
and place a $*$, and so on.  
Likewise,
convert $\sigma^{odd}$ to a $*$-pattern in the usual way but placing the
$*$'s only in the {\it odd boxes}.    
If
$\sigma^{odd}=[v_1,\ldots,v_r]$ then skip over $\frac{v_1-1}{2}$ odd
boxes and place a $*$, then skip over $\frac{v_2-v_1}{2}$ odd boxes
and place a $*$, and so on.  
This awkward procedure is best demonstrated by an
example, so return to $\sigma=[00123]$.  Then $\sigma^{ev}=[002]$, which
corresponds to the $*$-pattern $[\,\bblank * \bblank * \bblank \bblank \bblank *]$, and $\sigma^{odd}=[13]$
which corresponds to the $*$-pattern $[*\bblank \bblank \bblank *]$.  So the combined
pattern is $[** \bblank ** \bblank \bblank *]$.

We have given a function from partitions with $k$ pieces to
$*$-patterns with $k$ asterisks.  It is easy to see that this is a
bijection; an example of the inverse should suffice.  For the
$*$-pattern
\[ [* \bblank \bblank * \bblank * * * \bblank \bblank \bblank  * \bblank * \bblank *]
\]
the only odd boxes occupied are $1$ and $7$.  The associated
$\sigma^{odd}$ is $[15]$, because $\frac{5-1}{2}$ accounts for the two
skipped odd boxes between them.  The occupied even boxes are
$4,6,8,12,14,16$ and so $\sigma^{ev}=[222444]$.  The partition
associated to the above $*$-pattern is therefore $\sigma=[12224445]$.  

The point of this strange bijection is the following: it carries the
successor relation for $*$-patterns to the successor relation for
partitions (the latter defined back in Section~\ref{se:add-basis}).
This is easy to see, and we leave it to the reader---but also see
Remark~\ref{re:strange-bijection} below for a strong hint.

Using the above bijection, the minimal $*$-patterns of $k$ asterisks
correspond to the partitions 
$[00\ldots 0]$, $[00\ldots 01]$, $[00\ldots
001]$,$\ldots$, and $[11\ldots 1]$ (each with $k$ pieces).  
For example, if $k$ is even then the $*$-pattern with $a$-sequence $(1,2,3,\ldots,k)$
corresponds to the partition $[00\ldots 011\ldots 1]$ where there are
$\frac{k}{2}$ zeros and $\frac{k}{2}$ ones.  
It is somewhat better to order the partitions as 
\begin{myequation}
\label{eq:min-parts}
\qquad\qquad [00\ldots 0], \ [11\ldots 1], \ [00\ldots 01], \ [011\ldots 1], \ [00\ldots
001], \ [0011\ldots 1], \ \ldots 
\end{myequation}
because in this order the topological degrees of the associated
Schubert cells are 
\[ \tbinom{k+1}{2}, \tbinom{k}{2}, \tbinom{k-1}{2} ,\ldots,
\tbinom{2}{2}, \tbinom{1}{2}.
\]

For later use, let $\mu(i)$ be the number of $1$'s in the $i$th
partition from the list (\ref{eq:min-parts}), with $1\leq i \leq k+1$.  This
sequence is $\mu(1)=0$, $\mu(2)=k$, $\mu(3)=1$,
$\mu(4)=k-1$, and so forth.  Note that $\mu(i)=\gamma_{k+2-i}$, for
the $\gamma$-function defined in Proposition~\ref{pr:cell-bound}.

We can now wrap up the argument.  We have a bijection between
$*$-patterns and partitions, and it 
preserves the successor relation; it therefore also preserves the trees of
descendants.
In both settings (of $*$-patterns and partitions) one finds exactly
$k+1$ minimal elements---and therefore $k+1$ trees.  The minimal
partitions
are the ones in
which each piece is either $0$ or $1$.  For the $*$-patterns we have
computed that the minimal elements correspond to cells of bidegree
$\bigl (\binom{i}{2},\binom{i}{2}
\bigr )$ for
$1\leq i\leq k+1$, and that the partition associated to this $*$-pattern has
exactly $\gamma_i$ pieces equal to $1$ (and the rest zeros).  
We also found that an $r$th successor of such a $*$-pattern has bidegree
$\bigl (\binom{i}{2}+2r,\binom{i}{2}+r\bigr )$.

Let $\sigma_i$ be the partition associated to the minimal $*$-pattern of
bidegree $\bigl (\binom{i}{2},\binom{i}{2}\bigr )$.  Then
$\rank^{\tbinom{i}{2}+2r,\tbinom{i}{2}+r}(X)$ is the number of $r$th
successors of this $*$-pattern, which is equal to the number of $r$th
successors of the partition $\sigma_i$.  But 
$\sigma_i$ contains exactly $\gamma_i$ odd numbers, so the successors
of $\sigma_i$ are the partitions with exactly $\gamma_i$ odd numbers.  The
sum of the numbers in $\sigma_i$ is equal to $\gamma_i$ (note that
$\sigma_i$ only contains $0$s and $1$s), and so the sum of the numbers
in an $r$th successor of $\sigma_i$ will be $\gamma_i+2r$.  One sees
in this way that the number of $r$th successors of $\sigma_i$ is equal
to $\partition_{2r+\gamma_i,\leq k}[\gamma_i]$.  This
completes the proof.
\end{proof}

\begin{remark}
\label{re:strange-bijection}
Let us return to the classical bijection between $*$-patterns and
partitions, via Schubert symbols.  We claim that moving an asterisk
one spot to the right corresponds to adding $1$ to an element of the
associated Schubert symbol.  Suppose that the $a$-sequence for the
$*$-pattern is $\ldots x,y,z,\ldots$ and that we are promoting $y$ to
$y+1$.  Clearly this does not effect the beginning or the end of the
Schubert symbol.  If the original Schubert symbol was
$\ldots,u,v,w\ldots$ then $v-u=y-x-1$ and $w-v=z-y-1$.  
The new Schubert symbol will be $\ldots,u,v',w',\ldots$ where
$v'-u=(y+1)-x-1$ and $w'-v'=z-(y+1)-1$.  Clearly this requires $v'=v+1$
and $w'=w$.

It is not true, however, that moving an asterisk {\it two\/} spots to the right
corresponds to adding $2$ to an element of the associated Schubert symbol.  
The whole point of the strange bijection from the above proof was to
create a situation where this does work, and the previous paragraph
suggests why treating the even and odd spots separately accomplishes
this.
\end{remark}

\section{Differentials in the cellular spectral sequence}
\label{se:differentials}

The main goal of this section is Kronholm's theorem
(Theorem~\ref{th:Kronholm} below), which to date is
our best tool for governing what happens inside the cellular spectral sequence.

\medskip
To begin, 
we give two examples demonstrating the kinds of differentials that can
appear in the cellular spectral sequence for $\Gr_k(\cU)$.  The first example
consists of the row of three pictures below.  In the leftmost picture
we have a page of the spectral sequence in which there are two
copies of $\M_2$, with generators in bidegrees $(a,b)$ and
$(a+3,b+4)$.  (Note that one will typically have many more than two
copies of $\M_2$, but we focus on this simple situation for
pedagogical purposes).  There is a differential (shown) that must be a $d_3$,
since it maps a class from filtration degree $a$ into one from
filtration degree $a+3$.  The differential is only drawn on the
{\it generator\/} of the first copy of $\M_2$, but the differentials in the
cellular spectral sequence are $\M_2$-linear: so the one that is drawn
implies several other evident differentials.  
  
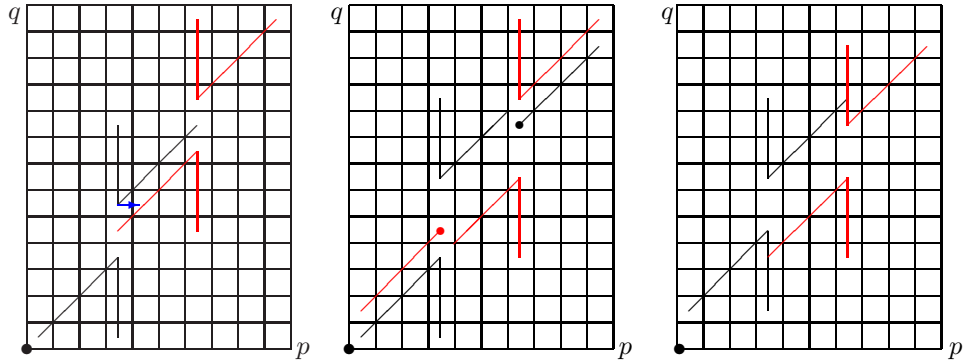
\begin{picture}(300,145)(5,0)
\multiput(0,0)(0,10){14}{\line(1,0){100}}
\multiput(0,0)(10,0){11}{\line(0,1){130}}
\put(102,-2){$p$}
\put(-7,125){$q$}
\put(0,0){\circle*{4}}
\put(34.5,54.5){\line(0,1){30}}
\put(34.5,54.5){\line(1,1){30}}
\put(34.5,34.5){\line(-1,-1){30}}
\put(34.5,34.5){\line(0,-1){30}}
\color[rgb]{1,0,0}
\put(64.5,94.5){\line(0,1){30}}
\put(64.5,94.5){\line(1,1){30}}
\put(64.5,74.5){\line(-1,-1){30}}
\put(64.5,74.5){\line(0,-1){30}}
\color[rgb]{0,0,1}
\put(34.5,54.5){\vector(1,0){8}}
\color[rgb]{0,0,0}
\multiput(122,0)(0,10){14}{\line(1,0){100}}
\multiput(122,0)(10,0){11}{\line(0,1){130}}
\put(224,-2){$p$}
\put(115,125){$q$}
\put(122,0){\circle*{4}}
\put(156.5,64.5){\line(0,1){30}}
\put(156.5,64.5){\line(1,1){25}}
\put(156.5,34.5){\line(-1,-1){30}}
\put(156.5,34.5){\line(0,-1){30}}
\put(186.5,84.5){\line(1,1){30}}
\put(186.5,84.5){\circle*{3}}
\color[rgb]{1,0,0}
\put(186.5,94.5){\line(0,1){30}}
\put(186.5,94.5){\line(1,1){30}}
\put(186.5,64.5){\line(-1,-1){25}}
\put(186.5,64.5){\line(0,-1){30}}
\put(156.5,44.5){\line(-1,-1){30}}
\put(156.5,44.5){\circle*{3}}
%
\color[rgb]{0,0,0}
\multiput(246,0)(0,10){14}{\line(1,0){100}}
\multiput(246,0)(10,0){11}{\line(0,1){130}}
\put(349,-2){$p$}
\put(240,125){$q$}
\put(247,0){\circle*{4}}
\put(280.5,64.5){\line(0,1){30}}
\put(280.5,64.5){\line(1,1){30}}
\put(280.5,44.5){\line(-1,-1){30}}
\put(280.5,44.5){\line(0,-1){30}}
\color[rgb]{1,0,0}
\put(310.5,84.5){\line(0,1){30}}
\put(310.5,84.5){\line(1,1){30}}
\put(310.5,64.5){\line(-1,-1){30}}
\put(310.5,64.5){\line(0,-1){30}}
\end{picture}

\vspace{0.2in}

In the middle panel we show the $E_{4}$-term of the spectral sequence,
obtained by taking homology with respect to our differential (warning:
not all
$\tau$-multiplications are shown here).  In our
simple example this is the same as $E_\infty$, but note that there are
extension problems in deducing the $\M_2$-structure.  By a theorem of
Kronholm \cite[Theorem 3.2]{K1} it turns out that the cohomology we
are converging to must be free over $\M_2$, and hence the extensions
are resolved as shown in the third panel.

Note the net effect as one passes from the first panel to the third:
the two copies of $\M_2$ remain, but their bidegrees have been
shifted.  The first copy has moved up one weight, and the second copy
has moved down one weight.  

Our next example shows a very similar phenomenon.  Interpreting the
pictures requires a little more imagination, though: remember that the
pictures only explicitly show the edges of the cones, whereas there are
an entire lattice of classes within the cones.  The leftmost chart
shows a situation where the differential takes the black generator to 
a class in the interior of the negative cone for the second copy of $\M_2$:

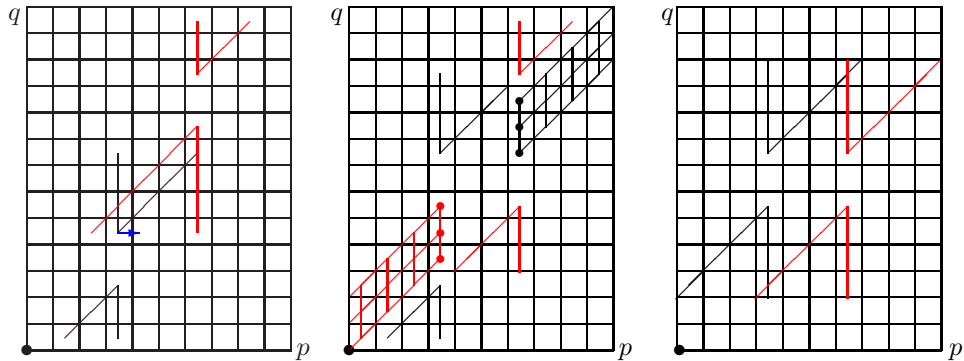
\begin{picture}(300,145)(5,0)
\multiput(0,0)(0,10){14}{\line(1,0){100}}
\multiput(0,0)(10,0){11}{\line(0,1){130}}
\put(102,-2){$p$}
\put(-7,125){$q$}
\put(0,0){\circle*{4}}
\put(34.5,44.5){\line(0,1){30}}
\put(34.5,44.5){\line(1,1){30}}
\put(34.5,24.5){\line(-1,-1){20}}
\put(34.5,24.5){\line(0,-1){20}}
\color[rgb]{1,0,0}
\put(64.5,104.5){\line(0,1){20}}
\put(64.5,104.5){\line(1,1){20}}
\put(64.5,84.5){\line(-1,-1){40}}
\put(64.5,84.5){\line(0,-1){40}}
\color[rgb]{0,0,1}
\put(34.5,44.5){\vector(1,0){8}}
\color[rgb]{0,0,0}
\multiput(122,0)(0,10){14}{\line(1,0){100}}
\multiput(122,0)(10,0){11}{\line(0,1){130}}
\put(224,-2){$p$}
\put(115,125){$q$}
\put(122,0){\circle*{4}}
\put(156.5,74.5){\line(0,1){30}}
\put(156.5,74.5){\line(1,1){25}}
\put(156.5,24.5){\line(-1,-1){20}}
\put(156.5,24.5){\line(0,-1){20}}
\multiput(186.5,94.5)(0,-10){3}{\circle*{3}}
\multiput(186.5,94.5)(10,10){4}{\line(0,-1){20}}
\multiput(186.5,94.5)(0,-10){3}{\line(1,1){35}}
\color[rgb]{1,0,0}
\put(186.5,104.5){\line(0,1){20}}
\put(186.5,104.5){\line(1,1){20}}
\put(186.5,54.5){\line(-1,-1){25}}
\put(186.5,54.5){\line(0,-1){25}}
\multiput(156.5,34.5)(0,10){3}{\circle*{3}}
\multiput(156.5,34.5)(0,10){3}{\line(-1,-1){34}}
\multiput(156.5,34.5)(-10,-10){4}{\line(0,1){20}}
%
\color[rgb]{0,0,0}
\multiput(246,0)(0,10){14}{\line(1,0){100}}
\multiput(246,0)(10,0){11}{\line(0,1){130}}
\put(349,-2){$p$}
\put(240,125){$q$}
\put(247,0){\circle*{4}}
\put(280.5,74.5){\line(0,1){35}}
\put(280.5,74.5){\line(1,1){35}}
\put(280.5,54.5){\line(-1,-1){35}}
\put(280.5,54.5){\line(0,-1){35}}
\color[rgb]{1,0,0}
\put(310.5,74.5){\line(0,1){35}}
\put(310.5,74.5){\line(1,1){35}}
\put(310.5,54.5){\line(-1,-1){35}}
\put(310.5,54.5){\line(0,-1){35}}
\end{picture}

\vspace{0.2in}

The leftmost chart is again an $E_3$-term, as the differential maps a
class in filtration $a$ to a class in filtration $a+3$.  The
$E_4$-page is shown in the second chart.  Kronholm's theorem tells us
that the cohomology our spectral sequence is converging to is free
over $\M_2$, and so the relevant extension problems work out to be as
shown in the third chart.

Once again, notice the difference between the first chart and the last
chart: the left copy of $\M_2$ has increased its weight by three,
whereas the right copy has decreased its weight by three.

Kronholm's theorem generalizes these two examples.  It says that 
the cohomology that the spectral sequence is converging to will be related
to the $E_1$-term by a sequence of ``trades'' in which two copies of
$\M_2$ shift up/down by the same number.  The following is a rigorous
statement along these lines, which covers all the applications we will
need in the present paper:

\begin{thm}[Kronholm]
\label{th:Kronholm}
Let $X$ denote the $E_1$-term of the cellular spectral sequence for $\Gr_k(\cU)$, and
let $Y=H^{*,*}(\Gr_k(\cU);\Z/2)$.  Both $X$ and $Y$ are free as
$\M_2$-modules, and
for each $p\in \Z$ one has
\[ \sum_{q} \rank^{p,q}(X) = \sum_q \rank^{p,q}(Y) 
\quad\text{and}\quad
 \sum_c \rank^{c,p+c}(X) = \sum_c \rank^{c,p+c}(Y).\]
\end{thm}

Note that the first equality from part (a) says that the number of basis elements in
topological dimension $p$ is the same in both $E_1$ and
$H^{*,*}(\Gr_k(\cU))$.  Relative to our rank charts, the second
equality from (a) 
says that the number of basis
elements along any given diagonal is the same in both $E_1$ and
$H^{*,*}(\Gr_k(\cU);\Z/2)$.  

\begin{remark}
In actuality, Theorem~\ref{th:Kronholm} as we have stated it is not quite found in 
\cite{K1}.  However, the result is implicit in the proof of
\cite[Theorem 3.2]{K1}. 
\end{remark}

\subsection{The forgetful map to singular cohomology}

There are natural maps $\Phi\colon H^{p,q}(X;\Z/2) \ra H^p(X;\Z/2)$
from equivariant cohomology to singular cohomology.  These fit
together to give a ring map $\Phi\colon H^{*,*}(X;\Z/2)\ra
H^*(X;\Z/2)$.  When $X=*$ this map is completely determined by the
formulas
$\Phi(\tau)=1$, $\Phi(\rho)=0$.
Consequently, $\Phi$ induces natural 
maps
\[ H^{*,*}(X;\Z/2)/(\rho) \ra H^*(X;\Z/2) \quad\text{and}\quad
H^{*,*}(X;\Z/2)[\tau^{-1}] \ra H^*(X;\Z/2).
\]

If $J$ is a free $\M_2$-module, then $J/\rho J$ is a free
$\M_2/\rho=\Z/2[\tau]$-module.  Note that $\tau$ has topological
dimension zero, and so $J/\rho J$ will decompose as a
$\Z/2[\tau]$-module into a direct sum over all topological dimensions:
\[ J/\rho J = \oplus_p \bigl [J/\rho J\bigr ]^{p,*}.
\]
Note that the submodule $[J/\rho J]^{p,*}$ is only ``influenced'' by basis
elements of $J$ in topological degree $p$: more precisely, any element
of $[J/\rho J]^{p,*}$ is the image under $J\ra J/\rho J$ of a $\Z/2[\tau]$-linear
combination of basis elements of $J$ in topological degree $p$.  
Also, if $J$ has a finite number of free generators in each topological
degree then the $\Z/2$-dimension of $[J/\rho J]^{p,q}$ is independent
of $q$ for $q\gg 0$ (once $q$ is larger than the weights of all the
generators in this topological degree).   

Let us apply these ideas when $J=H^{*,*}(\Gr_k(\cU);\Z/2)$.  Then
$[J/\rho J]^{p,*}$ is a free $\Z/2[\tau]$-module with a basis
corresponding to the equivariant Schubert cells of topological dimension
$p$.  While these cells likely have different weights, if we look in
$[J/\rho J]^{p,N}$ for $N$ large enough then we will see all of them
(more precisely, $\tau$-multiples of all of them).  The forgetful map
$\Phi$ will send these elements to the corresponding non-equivariant
Schubert classes in $H^p(\Gr_k(\cU);\Z/2)$ (recall that
$\Phi(\tau)=1$).  This shows that in large enough weights $N$ the map
$\Phi\colon [J/\rho J]^{p,N} \ra H^p(\Gr_k(\cU);\Z/2)$ is an
isomorphism.  This proves part (a) of the following:

\begin{prop} 
\label{pr:inject}  
For $p,q\in \Z$ 
consider the map
\[ \Phi_{p,q}\colon\Bigl
[H^{*,*}(\Gr_k(\cU);\Z/2)/(\rho)\Bigr ]^{p,q} \ra
H^p(\Gr_k(\cU);\Z/2).
\]
\begin{enumerate}[(a)]
\item Given $p$, there exists an $N\in \Z$ such that the map
$\Phi_{p,q}$ is an isomorphism for all $q\geq N$.
\item For any $p$ and $q$ the map $\Phi_{p,q}$ is an injection.
\end{enumerate}
\end{prop}

\begin{proof}
The proof of part (a) preceded the statement of the proposition.
For (b),
fix $p$ and $q$ and write $J=H^{*,*}(\Gr_k(\cU);\Z/2)$ for
simplicity.  
By (a) we know that for large enough $N$
the map $\Phi_{p,N}$
is
an isomorphism.  Now just consider the diagram
\[ \xymatrix{
[J/\rho J]^{p,q} \ar[r]^{\cdot \tau^{N-q}}\ar[dr]_{\Phi} & [J/\rho
J]^{p,N}\ar[d]_\iso^{\Phi} \\
& H^{p}(\Gr_k(\cU);\Z/2),
}
\]
which commutes because $\Phi(\tau)=1$.  Multiplication by $\tau$ is an
injection on $\M_2/(\rho)$, and hence also on $J/\rho J$.  So the
diagonal map in the diagram is also injective.
\end{proof}


\section{Proof of the main theorem}
\label{se:proof}

Throughout this section we let $X$ be the $E_1$-term of the cellular
spectral sequence for $\Gr_k(\cU)$, we let
$Y=H^{*,*}(\Gr_k(\cU);\Z/2)$, and we let $Z=\Inv_k$. 
It will be convenient to keep in mind the diagram
\[\xymatrix{
 X \ar@{~>}[r] & Y \ar[r] & Z,
}
\]
indicating that $Y$ maps to $Z$ and that there is a spectral sequence that
starts from $X$ and converges to $Y$.  Our aim is to prove
Theorem~\ref{th:main1}, stating that $Y\ra Z$ is an isomorphism.

Each of $X$, $Y$, and $Z$ is a free module over $\M_2$, and the proof will involve
a study of the bigraded rank functions for each.  The following lemma
collects the key results we will need:

\begin{lemma}
\label{le:main} \mbox{}\par
\begin{enumerate}[(a)]
\item For every $p\in \Z$, 
\[ \sum_q \rank^{p,q}(Y)=
\sum_q \rank^{p,q}(X)=\sum_q \rank^{p,q}(Z).
\]
\item For every $c\in \Z$, 
\[\sum_p \rank^{p,p+c}(Y)=
\sum_p \rank^{p,p+c}(X)=\sum_p \rank^{p,p+c}(Z).
\]
\item For every $p,q\in \Z$,
\[ \sum_{c\leq q} \rank^{p,c}(Y) \leq \sum_{c\leq q} \rank^{p,c}(Z)
\quad\text{and}\quad 
\sum_{c\geq q} \rank^{p,c}(Y) \geq \sum_{c\geq q} \rank^{p,c}(Z).
\]
\end{enumerate}
\end{lemma}

We have written the equalities in the first two parts in the order
that they will be proven: $Y$ is related to $X$, and $X$ is related to
$Z$.  
Phrase in terms of our rank charts, the above results say:
\begin{enumerate}[(i)]
\item The sum of the numbers in any column is the same for $X$, $Y$, and
$Z$.
\item The sum of the numbers along any diagonal is the same for $X$,
$Y$, and $Z$.
\item If one fixes a particular box and adds together the numbers in all boxes
directly above it, the sum for $Y$ is always at least the sum for
$Z$.  (This is the second inequality in (c)).
\end{enumerate}
We defer the proof of the lemma for just a moment, in order to
highlight the structure of the main argument.  However, let us point
out that the left equalities in (a) and (b) are by Kronholm's
Theorem, and the second equalities come from our combinatorial
analyses of $\rank^{*,*}(X)$ and $\rank^{*,*}(Z)$.  In light of (a), the two
inequalities in part (c) are equivalent.  The proof of these
final inequalities uses some topology, namely the non-equivariant version of
Theorem~\ref{th:main1}.

Before proving the next result, we introduce a useful piece of notation.
If $M$ is a free $\M_2$-module, then
for each $c\in \Z$ let $d_c(M)$ denote the function $\Z\ra \Z$ given by
$p\mapsto \rank^{p,p-c}(M)$.  These are the entries in the rank chart of
$M$ along the diagonal line of slope $1$ passing through the point
$(0,-c)$.

\begin{prop}
\label{pr:Y->Z}
For all $p,q\in \Z$, $\rank^{p,q}(Y)=\rank^{p,q}(Z)$.  
\end{prop}

\begin{proof}

We will prove the proposition by establishing that $d_c(Y)=d_c(Z)$ for
all $c\in \Z$.
First note that this is easy for $c<0$.  In this case we know by
direct computation that
$\rank^{p,p-c}(Z)=0$ for all $p\in \Z$ (Corollary~\ref{co:bounds}).  So $\sum_p
\rank^{p,p-c}(Z)=0$, which implies by
Lemma~\ref{le:main}(b) that $\sum_p \rank^{p,p-c}(Y)=0$.  Since the
ranks are all non-negative, this means $\rank^{p,p-c}(Y)=0$ for all
$p\in \Z$.  

Next we proceed by induction on $c$.  Assume $c\geq 0$ and that $d_n(Y)
=d_n(Z)$ for all $n< c$.  Let $p\in \Z$, and consider the inequality
\[ \sum_{q\geq p-c} \rank^{p,q}(Y) \geq \sum_{q\geq p-c}
\rank^{p,q}(Z) \]
from Lemma~\ref{le:main}(c).  By induction we know that
$\rank^{p,q}(Y)=\rank^{p,q}(Z)$ for $q>p-c$, and so we conclude that 
\begin{myequation}
\label{eq:eq1}
 \rank^{p,p-c}(Y) \geq \rank^{p,p-c}(Z).
\end{myequation}
This holds for all $p\in \Z$.  But we also know, by
Lemma~\ref{le:main}(b), that
\begin{myequation}
\label{eq:eq2} \sum_p \rank^{p,p-c}(Y) = \sum_p \rank^{p,p-c}(Z).
\end{myequation}
Equations (\ref{eq:eq1}) and (\ref{eq:eq2}) can both be true only if
 $\rank^{p,p-c}(Y)=\rank^{p,p-c}(Z)$
for all $p \in \Z$.  That is, $d_c(Y)=d_c(Z)$.  
\end{proof}

\begin{remark}
It is worth remarking that Proposition~\ref{pr:Y->Z} has solved one of our
main questions.  It completely identifies the weights of the free
generators for $H^{*,*}(\Gr_k(\cU);\Z/2)$ by showing that they agree
with the ranks of the generators for the combinatorially-computable ring of invariants
$\Inv_k$.  
\end{remark}

Next we give the

\begin{proof}[Proof of Lemma~\ref{le:main}]
The first equality in (a) is by Kronholm's Theorem
(Theorem~\ref{th:Kronholm}).  For the second equality observe that
$\sum_q
\rank^{p,q}(X)$ is just the number of classical Schubert
cells of dimension $p$ inside $\Gr_k(\R^\infty)$.  This is the same as
the number of partitions of $p$ into at most $k$ pieces, which is the
same as $\sum_q \partition_{p,\leq k}[q]$.  The latter equals $\sum_q
\rank^{p,q}(Z)$ by Proposition~\ref{pr:rank-Inv}.

For (b), the first equality is again by Kronholm's Theorem.  The
equality $\sum_p \rank^{p,p+c}(X)=\sum_p \rank^{p,p+c}(Z)$ follows
from the combinatorial identities in Proposition~\ref{pr:rank-Inv} and
Proposition~\ref{pr:cell-bound}; to see why, it is best to think
pictorially.
Proposition~\ref{pr:cell-bound} says that the rank chart for $X$
is concentrated along $k+1$ rays of slope $\frac{1}{2}$, emanating
from certain points on the $y=x$ line.  Proposition~\ref{pr:rank-Inv}
says that the rank chart of $Z$ also consists of $k+1$ rays of slope 
$\frac{1}{2}$---containing the same entries as the ones in $X$---but
which emanate from different points on the $y=x$ line (in other words,
the order of the rays in the two charts are both permuted and shifted
along the $y=x$ line).  From this it follows at once that the
diagonals of the two rank charts  contain the same entries,
only permuted.  In particular, the sum of the entries is the same in
the two situations.

For (c), it will suffice to prove that $\sum_{c\leq q} \rank^{p,c}(Y)
\leq \sum_{c\leq q} \rank^{p,c}(Z)$, since the second inequality
follows from this one together with part (a).    Consider the diagram
\[ \xymatrixcolsep{1pc}\xymatrix{
Y^{p,q} \ar[rr]\ar[dr]\ar[dd]_\Phi && Z^{p,q}\ar[dr]\ar[dd]^-<<<<<<\Phi \\
& [Y/\rho Y]^{p,q}\ar@{.>}[dl]\ar[rr] && [Z/\rho Z]^{p,q}\ar@{.>}[dl]\\ 
H^p(\Gr_k(\cU);\Z/2) \ar[rr]_\iso && \Bigl [ \bigl [H^*(\RP^\infty;\Z/2)^{\tens
k}\bigr ]^{\Sigma_k} \Bigr ]^{p}
}
\]
where the dotted arrows exist because $\rho$ is sent to zero by
$\Phi$.  The bottom horizontal map is an isomorphism by the classical
theory, and the map $[Y/\rho Y]^{p,q} \ra H^p(\Gr_k(\cU);\Z/2)$ is an
injection by Proposition~\ref{pr:inject}(b).  It follows that $[Y/\rho
Y]^{p,q} \ra [Z/\rho Z]^{p,q}$ is an injection.  However, it is easy
to see that if $J$ is a free $\M_2$-module then $\dim_{\Z/2} [J/\rho
J]^{p,q}=\sum_{c\leq q}
\rank^{p,c}(J)$.  Applying this to $Y$ and $Z$, we have completed the proof.
\end{proof}

At this point we have only proven that $Y$ and $Z$ are free
$\M_2$-modules with the same bigraded rank functions.  But we have a
specific map $Y\ra Z$, and our goal is to prove that it is an
isomorphism.  Since both $Y^{p,q}$ and $Z^{p,q}$ are
finite-dimensional over $\Z/2$ for every $p,q\in \Z$, it will be
sufficient to prove that $Y \ra Z$ is surjective.  We begin with the
following observation:

\begin{lemma}
\label{le:iso-mod-rho}
The map $Y/\rho Y \ra Z/\rho Z$ is an isomorphism.
\end{lemma}

\begin{proof}
As in the proof of Lemma~\ref{le:main}(c), we know that $[Y/\rho Y]^{p,q}\ra
[Z/\rho Z]^{p,q}$ is an injection.  We also know that the $\Z/2$-dimensions of these
two spaces are $\sum_{c\leq q} \rank^{p,c}(Y)$ and $\sum_{c\leq q}
\rank^{p,c}(Z)$, which are equal by Proposition~\ref{pr:Y->Z}.  This
proves the lemma.
\end{proof}

The desired result will now follow from the purely algebraic lemma
below:

\begin{lemma}
\label{le:algebra}
Let $M$ and $N$ be free $\M_2$-modules, and let $f\colon M\ra N$ be a
map such that $M/\rho M \ra N/\rho N$ is an isomorphism.  Assume that
\begin{enumerate}[(i)]
\item $\rank^{p,q}(M)=\rank^{p,q}(N)$ for all $p,q\in \Z$.
\item $\dim_{\Z/2} M^{p,q}$ is finite for all $p,q\in \Z$.  
\item There exists an $r\in \Z$ such that $d_c(M)=d_c(N)=0$ for all
$c< r$.
\item There exists a number $u$ such that $\rank^{p,q}(M)=0$ for all
$p< u$.  
\end{enumerate}
Then $f$ is an isomorphism.
\end{lemma}

\begin{proof}
Pick a free basis $\{e_\alpha\}$ for $N$  consisting of homogeneous
elements.  For each $s\in \Z$ let $N_s\subseteq N$ be the submodule spanned
by all $e_\alpha$ for which the bidegree $(p_\alpha,q_\alpha)$
satisfies $p_\alpha-q_\alpha \leq
s$ (these are the basis elements on all diagonals `higher than' the
$p-q=s$ diagonal).   
Note that $N_s=0$ for $s<r$, where $r$ is the
number specified in condition (iii).

Condition (iv) readily implies the following fact: for every
$p,q\in \Z$ there exists an $m\geq 0$ such that $[\rho^m
N]^{p,q}\subseteq N_{p-q-1}$.  In other words, every element of
$N^{p,q}$ that is a multiple of $\rho^m$ is in the $\M_2$-span of
basis elements from higher diagonals.  (One need only take $m=p-u+1$
here, where $u$ is from condition (iv)).  

We will prove by induction that each $N_s$ is contained in the image
of $f$.  We know this for $s<r$ since in that case $N_s=0$.  So assume $s\in
\Z$ and $N_{s-1}\subseteq \im f$. 
Since $M/\rho M \ra N/\rho N$ is an isomorphism it follows that $N=(\im
f) + \rho N$.  Substituting this equation for $N$ into itself, we then
find 
\[ N = (\im f)+\rho N = (\im f)+\rho^2 N =(\im f)+\rho^3 N = \cdots 
\]
So $N=(\im f)+\rho^n N$
for any $n\geq 1$.  

Now let $e_\alpha$ be a basis element lying in $N_s$, of bidegree
$(p,q)$ (so that $p-q\leq s$).  We may assume $p-q=s$, for
otherwise $e_\alpha \in N_{s-1}$ and so is in the image of $f$ by
induction.  By the second paragraph of this proof, there exists $m\geq
1$ such that $[\rho^m N]^{p,q} \subseteq N_{s-1}$.  But then we have
\[ N^{p,q}=(\im f)^{p,q} + [\rho^m
N]^{p,q} \subseteq (\im f) + N_{s-1} =\im f
\]
where the last equality uses our inductive assumption that
$N_{s-1}\subseteq \im f$.  
We have therefore shown that $e_\alpha \in \im f$, and since this
holds for every basis element we have $N_s\subseteq \im f$.

At this point we have shown that $f$ is surjective.
The finiteness condition (ii) then implies that $f$ is indeed an isomorphism.
\end{proof}

We now restate Theorem~\ref{th:main1} from the introduction, and tie up
its proof:

\begin{thm}
The map  $\eta^*\colon H^{*,*}(\Gr_k(\cU);\Z/2) \ra \bigl
[H^{*,*}(\Gr_1(\cU);\Z/2)^{\tens k }\bigr ]^{\Sigma_k}$
is an isomorphism of bigraded rings.
\end{thm}

\begin{proof}
This is the map $Y\ra Z$ considered throughout this section.  Both $Y$
and $Z$ are free $\M_2$-modules that satisfy hypotheses (ii)--(iv) of
Lemma~\ref{le:algebra}.  Proposition~\ref{pr:Y->Z} verifies condition
(i) of that lemma.  
The result therefore follows by that lemma
together with Lemma~\ref{le:iso-mod-rho}.
\end{proof}


\section{The multiplicative structure of the ring of invariants}
\label{se:mult}

At this point in the paper we have proven that our map 
\[ \eta^*\colon
H^{*,*}(\Gr_k(\cU)) \ra [H^{*,*}(\Gr_1(\cU))^{\tens k}]^{\Sigma_k}
\] 
is an isomorphism of rings.  We also have a combinatorial description
of the bigraded rank function---that is, we understand the additive
structure of these rings, or their structure as $\M_2$-module.  In this section we
further investigate the ring of invariants, concentrating on the
multiplicative structure.  Recall that this ring of invariants is
denoted $\Inv_k$ for short.

\medskip

\subsection{First observations}
Recall from Section~\ref{se:add-basis} that we use the notation
$w_i=[a_1\ldots a_i]$ and $c_i=[b_1\ldots b_i]$.  These are the $i$th
elementary symmetric functions in the $a$'s and $b$'s, respectively.  
More generally, define the invariant element $\wc_{i,j}$ by
\[ \wc_{i,j}=[a_1\ldots a_i b_{i+1}\cdots b_{i+j}].
\]
Note that this only makes sense when $i+j\leq k$. Note also that
$\wc_{i,0}=w_i$ and $\wc_{0,j}=c_j$.  Finally, let us observe that the
bidegree of $\wc_{i,j}$ is $(i,i)+(2j,j)=(i+2j,i+j)$.  

As  a warm-up for our investigation let us consider some basic relations.
The easiest relation one encounters is 
\[ w_1^2=(a_1+\cdots+a_n)^2=a_1^2+\cdots+a_n^2 =
(\rho a_1+\tau b_1)+\cdots + (\rho a_n+\tau b_n) = \rho w_1 +\tau c_1.
\]
Analogously, 
\begin{align*}
w_2^2=[a_1a_2]^2 &=\sum_{i< j} (\rho a_i+\tau b_i)(\rho a_j+\tau
b_j)\\
&=\rho^2 \sum_{i<j}a_ia_j + \rho\tau \sum_{i<j}(a_ib_j+a_jb_i) + \tau^2
\sum_{i<j} b_ib_j
\\
& =\rho^2 w_2 + \rho\tau\sum_{i\neq j} a_ib_j + \tau^2 c_2 \\
&= \rho^2 w_2 + \rho\tau \cdot \wc_{1,1} + \tau^2 c_2.
\end{align*}
More generally we have the following (the proof is left as an exercise):

\begin{prop}
In $\Inv_k$ there is the relation
\[w_j^2=\tau^j c_j + \tau^{j-1}\rho \wc_{1,j-1} +
\tau^{j-2}\rho^2 \wc_{2,j-2} +\cdots + \tau\rho^{j-1} \wc_{j-1,1} + \rho^j
w_j
\]
for any $j\leq k$.
\end{prop}

Next let us consider the products $w_1w_i$ for
various $i$.  For instance, $w_1w_2 = (a_1+\ldots+a_n)(a_1a_2+\ldots +
a_{n-1}a_n)$.  When we distribute, we will get terms that look like
$a_1^2a_2$, and also terms that look like $a_1a_2a_3$.  Note that the
former term only appears once, whereas the latter appears
$\binom{3}{2}=3$ times (which is equivalent to once, since we are in
characteristic two).  So we can write
\[ w_1w_2=[a_1^2a_2] + [a_1a_2a_3]=[a_1^2a_2]+w_3.\]
We must be careful when identifying $[a_1^2a_2]$.  We have
\[ [a_1^2a_2]=\sum_{i\neq j} a_i^2a_j =\sum_{i\neq j} (\rho a_i+\tau
b_i) a_j = \rho \sum_{i\neq j} a_ia_j + \tau \sum_{i\neq j} b_ia_j =
0 + \tau \wc_{1,1}.
\]  
Note that $\sum_{i\neq j} a_ia_j =0$ only because we are in
characteristic $2$.  

As one more example, let's compute $w_1w_3$.  We are looking at
the product $(a_1+\ldots+a_n)(a_1a_2a_3+\ldots )$, and so we have
terms that look like $a_1^2a_2a_3$ and $a_1a_2a_3a_4$.  The former
occurs exactly once, the latter $\binom{4}{1}=4$ times (equivalent to
zero times, mod $2$).  So 
\[ w_1w_3=[a_1^2a_2a_3]=\sum_{{j <k} \atop {i\notin \{j,k\}}} 
a_i^2a_ja_k 
=\sum_{{j <k} \atop {i\notin \{j,k\}}} 
(\rho a_i + \tau b_i) a_ja_k
= \rho[a_1a_2a_3]+\tau[a_1a_2b_3].
\]
The last equality takes a little thought: we must ask ourselves how many
times a typical term $a_1a_2a_3$ appears in the 
sum $\sum\limits_{{j <k} \atop {i\notin \{j,k\}}} 
a_ia_ja_k$, and the answer is that it occurs exactly three times
(equivalent to once, mod $2$).

The following proposition is easily proven by the above techniques:

\begin{prop}
\label{pr:w-decomp}
In $\Inv_k$ one has the relations $w_1w_{2i}=\tau \cdot \wc_{2i-1,1} +
w_{2i+1}$ and $w_1w_{2i+1}=\tau \wc_{2i,1}+\rho w_{2i+1}$.
\end{prop}
  
Note that the first relation from Proposition~\ref{pr:w-decomp} shows
that
$w_{2i+1}$ is decomposable in $\Inv_k$.  Without much trouble this
generalizes to the following result.  Compare \cite[Remark 3.4]{M}.

\begin{prop}  
\label{pr:w-indecomposable}
Let $1\leq j\leq k$.  Then
$w_j$ is indecomposable in $\Inv_k$ if and only if $j$ is a power of $2$.
\end{prop}

\begin{proof}
If $j$ is not a power of $2$, then $\binom{j}{i}$ is odd for some $i$.
Consider the product 
\[ w_i w_{j-i}=(a_1a_2\ldots a_i + \cdots)(a_1a_2\ldots a_{j-i}+
\cdots).
\]
When we distribute, we have some terms which contain one or more
squares---these belong to the ideal $(\rho,\tau)$ of $\Inv_k$ because of
the relation  $a_i^2=\rho a_i + \tau b_i$.  A typical term which
doesn't involve squares is $a_1a_2\ldots a_j$, and this appears
exactly $\binom{j}{i}$ times in the big sum.  So we can write
\[ w_i w_{j-i} \in (\rho,\tau) + w_{j}.
\]
But the elements of $(\rho,\tau)$ are by nature decomposable, and so
we have that $w_j$ is decomposable.

For the proof that $w_{2^r}$ is indecomposable, we map our ring
$\Inv_k$ to a simpler ring where it is easier to prove this.  
Specifically, consider the map
\[ \M_2[a_1,\ldots,a_k,b_1,\ldots,b_k]/(a_i^2=\rho a_i+\tau b_i) \ra
\Lambda_{\Z/2}(a_1,\ldots,a_k) \]
that sends $\rho$, $\tau$, and all the $b_i$'s to zero.  Upon taking
invariants this gives a map
\[ \Inv_k \ra \Lambda_{\Z/2}(a_1,\ldots,a_k)^{\Sigma_k}\]
that sends each $w_i$ to the $i$th symmetric function $\sigma_i$ in the $a_j$'s.  
But in
$\Lambda_{\Z/2}(a_1,\ldots,a_k)^{\Sigma_k}$ it is well-known that
$\sigma_i$ is indecomposable when $i$ is a power of $2$ (see
Proposition~\ref{pr:exterior-inv} below for a proof).    
\end{proof}

\subsection{Generalized Stiefel-Whitney classes}
One of the difficulties in studying the ring $\Inv_k$ is that there
does not seem to be a clear choice of which algebra generators to use;
every choice seems to have drawbacks.
The $\wc$ classes defined above represent one extreme: they result
from making the indices on the $a$'s and $b$'s disjoint.  The opposite
approach is to make the indices overlap as much as possible, and that
leads to the following definition:
\[ w_i^{(e)}=[a_1\ldots a_ib_1^e\ldots b_i^e].
\]
Note that this defines an element of $\Inv_k$ for $1\leq i\leq k$ and
$0\leq e$.  It has bidegree $(i,i)+ei(2,1)=(i(2e+1),i(e+1))$, and in
terms of our rank charts it lies on the same line of slope
$\frac{1}{2}$ as the class $w_i$.  Notice that $w_i^{(0)}=w_i$.

\subsection{Indecomposables}
Let $\epsilon \colon \M_2[a_1,\ldots,a_k,b_1,\ldots,b_k]/(a_i^2=\rho
a_i+\tau b_i) \ra \M_2$ be defined by sending each $a_i$ and $b_i$ to
zero.  We will also write $\epsilon$ for the restriction to $\Inv_k$.
Let $I_k\subseteq \Inv_k$ be the kernel of $\epsilon \colon \Inv_k \ra
\M_2$.  Then $I_k/I_k^2$ is a bigraded $\M_2$-module that is readily
checked to be free; it is called the
\dfn{module of indecomposables} for $\Inv_k$ relative to $\M_2$.  Our goal is to
determine the bigraded rank function for $I_k/I_k^2$, as well as a
basis.  In other words, we aim to write down a complete set of
representatives for the indecomposables in $\Inv_k$.  

\begin{remark}
It is worth stressing that we have set things up so that
`indecomposable' means {\it relative to\/} $\M_2$.  The elements
$\rho$, $\tau$, and $\theta$ are of course indecomposable elements of
$\Inv_k$ in the `absolute' sense, but we do not want to keep track of
them.  They will not be reflected in the rank function for $I_k/I_k^2$, which
by definition  counts the number of basis elements over $\M_2$.
\end{remark}

The main result is as follows:

\begin{thm} \mbox{}\par
\label{th:indecomp}
\begin{enumerate}[(a)]
\item
The indecomposables of $\Inv_k$ are represented by the classes
$c_1,\ldots,c_k$ together with the classes $w_{2^i}^{(e)}$ for $1\leq
2^i \leq k$ and $0\leq e \leq \frac{k}{2^i}-1$.  That is to say, these
classes give a free basis for $I_k/I_k^2$ as an $\M_2$-module.
\item The number of indecomposables for $\Inv_k$ is 
\[ 3k-(\text{\# of ones in the binary expansion of $k$}).
\]
\item For $1\leq 2^i\leq k$ and $0\leq e\leq \frac{k}{2^i}-1$
the classes $\wc_{2^i,e2^i}$ and $w_{2^i}^{(e)}$ 
are equivalent modulo decomposables.
\item 
For $p,q\in \Z$, $\rank^{p,q}(I_k/I_k^2)=0$ unless $0\leq
p$ and $0\leq q\leq k$.  
\item  When $p$ is odd and $0\leq p$,  
\[ \rank^{p,q}(I_k/I_k^2)=\begin{cases} 1 & \text{if $q=\frac{p+1}{2}$},\\
0 & \text{otherwise}.
\end{cases}
\]
The unique indecomposable in topological dimension $p$ is represented
by $w_1^{(\frac{p-1}{2})}$, or equivalently by $\wc_{1,\frac{p-1}{2}}$.
\item When $p$ is even and positive, write $p=2^i(2e+1)$.  Then
\[ \rank^{p,q}(I_k/I_k^2)=\begin{cases} 1 & \text{if $q=\frac{p}{2}$
or $q=\frac{p}{2}+2^{i-1}$},\\
0 & \text{otherwise}.
\end{cases}
\]
When $q=\frac{p}{2}$, the unique indecomposable in bidegree $(p,q)$ is represented by the
Chern class $c_{q}$.  When $q=\frac{p}{2}+2^{i-1}$ the unique
indecomposable is represented by $w_{2^i}^{(e)}$, or equivalently by
$\wc_{2^i,e\cdot 2^i}$.  
\end{enumerate}
\end{thm}

To paraphrase the above theorem, in the limiting case $k\ra\infty$ 
there is one indecomposable in every
odd topological dimension and two indecomposables in every even
topological dimension.  The following chart shows the exact
bidegrees, with different symbols for different types of
indecomposables:

\begin{picture}(300,180)(-50,-10)
\multiput(0,0)(0,15){11}{\line(1,0){240}}
\multiput(0,0)(15,0){17}{\line(0,1){150}}
\multiput(7.5,7.5)(30,15){8}{\circle{5}}
\multiput(19,18.5)(30,15){8}{$\square$}
\multiput(21,20.5)(30,15){8}{$\scriptscriptstyle{1}$}
\multiput(36,35.5)(60,30){4}{$\scriptscriptstyle{2}$}
\multiput(66,65.5)(120,60){2}{$\scriptscriptstyle{4}$}
\put(126,125.5){$\scriptscriptstyle{8}$}
\put(34,34){$\square$}
\put(64,64){$\square$}
\put(124,124){$\square$}
\put(94,64){$\square$}
\put(154,94){$\square$}
\put(184,124){$\square$}
\put(214,124){$\square$}
\end{picture}

\noindent
The circles represent the Chern classes, whereas the squares represent
the $w$-classes.  The squares with an $i$ inside represent $w_i^{(e)}$
classes, for $0\leq e$.  The pattern here is that the $w_i^{(e)}$
classes start in bidegree $(i,i)$ and then proceed up along the line
of slope $\frac{1}{2}$, occuring every $i$ steps along this line,
where ``step'' means a $(2,1)$ move. 

For $\Inv_k$ one cuts the chart off and only
takes the classes in weights less than or equal to $k$.  For example,
in $\Inv_5$ there will be the following
indecomposables (given in order of increasing topological degree):
\[ w_1,\ c_1,\ w_2,\ w_1^{(1)},\ c_2,\ w_4,\ w_1^{(2)},\ c_3,\
w_2^{(1)},\ w_1^{(3)},\ c_4,\ w_1^{(4)},\ c_5.
\]
Note that Theorem~\ref{th:indecomp}(b) predicts the number of
indecomposables to be $15-2=13$, which agrees with the above list.

\vspace{0.1in}

Our goal is now to prove Theorem~\ref{th:indecomp}, proceeding by a
series of reductions.  

\begin{proof}[Proof of Theorem~\ref{th:indecomp}]
The complexities of $\M_2$ are  irrelevant to the
considerations at hand.  To this end, define
$R_k=\Z/2[\tau,\rho,a_1,\ldots,a_k,b_1,\ldots,b_k]/(a_i^2=\rho
a_i+\tau b_i)$.  Let $S_k=R_k^{\Sigma_k}$, where the $\Sigma_k$-action
permutes the $a_i$'s and $b_i$'s but fixes $\rho$ and $\tau$.  
Let $\epsilon\colon R_k \ra \Z/2[\tau,\rho]$ be the map that sends 
$a_i$ and $b_i$ all to zero, for $1\leq i \leq k$.  Let $J_k$
be the augmentation ideal of $S_k$, defined as
\[ J_k=\ker (S_k \ra R_k \llra{\epsilon} \Z/2[\tau,\rho]).
\]
It is easy to see that $\Inv_k\iso S_k\tens_{\Z/2[\tau,\rho]} \M_2$
and $I_k/I_k^2 \iso (J_k/J_k^2)\tens_{\Z/2[\tau,\rho]} \M_2$.  So 
the bigraded rank function for $J_k/J_k^2$ over
$\Z/2[\tau,\rho]$ coincides with the bigraded rank function for $I_k/I_k^2$ 
over $\M_2$.  It will therefore suffice for us to prove the theorem in
the former case.  

A free basis for $J_k/J_k^2$ over $\Z/2[\tau,\rho]$ is the same as a
vector space basis for $J_k/[J_k^2+(\rho,\tau)J_k]$ over $\Z/2$.  
This is the form in which we will study the problem.  

Let $\tilde{R}_k=\Z/2[a_1,\ldots,a_k,b_1,\ldots,b_k]/(a_i^2)$ with the
evident $\Sigma_k$-action, and let $\tilde{S}_k=\tilde{R}_k^{\Sigma_k}$.
Consider the diagram
\[ \xymatrix{
J_k \ar[d]\ar@{ >->}[r] & S_k \ar[d]\ar[r]^-\epsilon & \Z/2[\tau,\rho]
\ar[d] \\
\tilde{J}_k \ar@{ >->}[r] & \tilde{S}_k \ar[r]^{\tilde{\epsilon}} & \Z/2
}
\]
where the vertical maps send $\rho$ and $\tau$ to zero, and
$\tilde{J}_k$ is the kernel of $\tilde{\epsilon}$.  It is easy to see
that $S_k\ra \tilde{S_k}$ is surjective: a $\Z/2$-basis for the target
is given by the orbit sums $[m]$ where $m$ is a monomial in the $a$'s
and $b$'s, and such an orbit sum lifts into $S_k$.  The same argument
shows that $J_k\ra \tilde{J}_k$ is surjective.  We in fact have a
surjection 
\[ J_k/[J_k^2+(\rho,\tau)J_k] \fib \tilde{J}_k/\tilde{J}_k^2, 
\]
and it is easy to see that this is actually an isomorphism.
  
We have therefore reduced our problem to understanding the module of
indecomposables $\tilde{J}_k/\tilde{J}_k^2$ for the ring
$\tilde{S}_k$.   This is a fairly routine algebra problem; we give a
full treatment in Appendix A for lack of a suitable reference.  See
Theorem~\ref{th:En} for the classification of the indecomposables,
proving parts (a) and (b).  The third statement in Lemma~\ref{le:in3} proves part
(c), and parts (d)--(f) are really just
restatements of (a) and (c).
\end{proof}

\subsection{Relations}
\label{se:relations}
In general it seems that writing down a complete set of relations
for $\Inv_k$ is not practical or useful.  See the cases of $k=2$ and
$k=3$ described in the next section.  The relations tend to be
numerous and also fairly complicated.  One general remark worth making
is that there will always be a relation for the square of a
$w_i^{(e)}$ class.  The square of $[a_1\ldots a_ib_1^e\ldots b_i^e]$ will be
$[a_1^2\ldots a_i^2b_1^{2e}\ldots b_i^{2e}]$, and each $a_j^2$
decomposes as $\rho a_j+\tau b_j$.  For example,
\begin{myequation}
\label{eq:w-square} \Bigl [w_1^{(e)}\Bigr
]^2=[a_1^2b_1^{2e}]=\rho[a_1b_1^{2e}]+\tau[b_1^{2e+1}] =\rho
w_1^{(2e)}+\tau [b_1^{2e+1}].
\end{myequation}
To express this in terms of indecomposables we need to write the power
sum $[b_1^{2e+1}]$ as a polynomial in the elementary symmetric
functions, via the mod $2$ Newton polynomials.  This already produces
an expression with lots of terms.  If $2e>k-1$ then $w_1^{(2e)}$ is
not an indecomposable and we also need to rewrite that term.  This can
be handled via the following result:

\begin{lemma}
\label{le:w1e}
In $\Inv_k$ one has the relation 
\[ w_1^{(e)}=w_1^{(e-1)}c_1 + w_1^{(e-2)}c_2+\cdots + w_1^{(e-k)}c_k
\]
for any $e\geq k$.
\end{lemma}

\begin{proof}
This follows from the identities
\begin{align*}
&[a_1b_1^e]=[a_1b_1^{e-1}]\cdot [b_1] + [a_1b_1^{e-1}b_2] \\
&[a_1b_1^{e-1}b_2]=[a_1b_1^{e-2}]\cdot [b_1b_2]+[a_1b_1^{e-2}b_2b_3]\\
&\vdots
\end{align*}
We stop when the right-hand term is $[a_1b_1^{e-(k-1)}b_2\ldots b_k]$,
since in this case the monomial $b_1\cdots b_k$ is a common factor to
all the summands in the $\Sigma_k$-orbit and can be taken out:
\[ [a_1b_1^{e-(k-1)}b_2\ldots b_k]=[a_1b_1^{e-k}]\cdot [b_1\cdots
b_k]=w_1^{(e-k)}\cdot c_k.
\]
Substituting each identity into the previous one leads  to
the desired relation.
\end{proof}

Let us work through one example.  In $\Inv_3$ there is the
indecomposable $w_1^{(2)}$, and according to our above analysis its
square is
\begin{myequation}
\label{eq:newt}
\quad\qquad \Bigl [w_1^{(2)}\Bigr ]^2=\rho w_1^{(4)} + \tau[b_1^5]=\rho
w_1^{(4)} + \tau[c_1^5 + c_1c_2^2+ c_1^2c_3 + c_1^3c_2+c_2c_3].
\end{myequation}
The latter expression comes from working out the appropriate Newton polynomial.
For the $w_1^{(4)}$ term we have
\[
w_1^{(4)}  =w_1^{(3)}c_1 + w_1^{(2)}c_2 + w_1^{(1)}c_3
= \bigl [ w_1^{(2)}c_1 + w_1^{(1)}c_2 + w_1 c_3\bigr ] c_1 +
w_1^{(2)}c_2 + w_1^{(1)}c_3 
\]
by two applications of Lemma~\ref{le:w1e}.  Our final relation is
\[ \Bigl[ w_1^{(2)}\Bigr ]^2 = \rho \bigl [ w_1^{(2)}(c_1^2 +c_2) +
w_1^{(1)}(c_1c_2+c_3) + w_1 c_1c_3 \bigr ] + \tau
[c_1^5 + c_1c_2^2+ c_1^2c_3 + c_1^3c_2+c_2c_3].
\]
This gives a fair indication of the level of awkwardness to this approach.

\subsection{The stable case}
The ring of invariants $\Inv_k$ will typically require many relations
beyond just those for the squares on the $w$-classes---see the
examples in Section~\ref{se:examples}.  However, things become simpler
in the stable case $k\ra\infty$.  We describe this next.

Recall that $T_k=\M_2[a_1,\ldots,a_k,b_1,\ldots,b_k]/(a_i^2=\rho
a_i+\tau b_i)$.  The map $T_{k+1}\ra T_k$ that sends $a_{k+1}$ and
$b_{k+1}$ to $0$ induces a surjection $\Inv_{k+1}\ra \Inv_k$ which is
an isomorphism in topological degrees less than $k+1$ (the latter is
immediate from looking at the standard free bases over $\M_2$).  Write
$\Inv_\infty$ for the inverse limit of
\[ \cdots \fib \Inv_3 \fib \Inv_2 \fib \Inv_1
\]
From Theorem~\ref{th:indecomp} it follows that the indecomposables of
this ring are the classes $c_j$ for $1\leq j$ and the classes
$w_{2^i}^{(e)}$ for $0\leq i$ and $0\leq e$.  

\begin{prop}
\label{pr:stable}
There exist a collection of polynomials $R_{i,e}$ such that
$\Inv_\infty$ is the quotient of $\M_2[c_j,w_i^{(e)}\,|\,i,j,e\in
\Z_{\geq 0}]$ by the relations
\[ \Bigl [ w_i^{(e)} \Bigr ]^2 = R_{i,e}.
\]
\end{prop}

\begin{remark}
Unfortunately the polynomials $R_{i,e}$ seem cumbersome to work out in
general.  We saw in (\ref{eq:newt}) that $R_{1,e}=\rho
w_1^{(2e)}+\tau[N_{2e+1}(c_1,\ldots))]$
where $N_{2e+1}$ is the mod $2$ Newton polynomial for writing the
$(2e+1)$-power sum as a polynomial in the elementary symmetric
functions.
The polynomial $R_{2,e}$ is more unpleasant; it has the form
\[ R_{2,e}=\rho^2 w_2^{(2e)} + \rho\tau \Bigl [ w_1^{(4e+1)} +
w_1^{(2e)} N_{2e+1}(c_1,\ldots) \Bigr ] + \tau^2[b_1^{2e+1}b_2^{2e+1}]
\]
where the expression  $[b_1^{2e+1}b_2^{2e+1}]$
must be replaced by a certain complicated, Newton-like polynomial in the Chern
classes.
\end{remark}

\begin{proof}[Proof of Proposition~\ref{pr:stable}]
We let $R_{i,e}$ be the polynomials constructed as in
Section~\ref{se:relations}---it is clear enough that they exist, it is
just not clear how to write down their coefficients in a reasonable way.
Consider the surjection
\[ \M_2[c_j,w_i^{(e)}\,|\,i,j,e\in
\Z_{\geq 0}]/(R_{i,e}) \fib \Inv_\infty.
\]
We claim that the bigraded Poincar\'e series for these two algebras are
identical, and from this it immediately follows that the map is an
isomorphism.  Both the domain and target are free $\M_2$-modules, so
it suffices to instead look at the bigraded rank functions.  

The domain has a free $\M_2$-basis consisting of
monomials in the variables $c_j$ and $w_i^{(e)}$ that are square-free
in the $w$-classes.  So the bigraded rank function is the same as for the
algebra
\[ \Lambda(w_i^{(e)}\,|\,0\leq i,0\leq e)\tens \F_2[c_1,c_2,\ldots].
\]
Likewise, the bigraded rank function for $\Inv_\infty$ is the same as
the Poincar\'e series for the algebra $L_\infty$ from Appendix A ($L_\infty$
is just the quotient of $\Inv_\infty$ obtained by killing $\rho$ and
$\tau$).  
But Theorem~\ref{th:En}(c) gives the isomorphism of graded rings $L_\infty\iso
\Lambda(w_i^{(e)}\,|\,0\leq i,0\leq e)\tens \F_2[c_1,c_2,\ldots]$, so
this completes the proof.
\end{proof}

\section{Examples}
\label{se:examples}

Our purpose in this section is to take a close look at
$H^{*,*}(\Gr_2(\cU);\Z/2)$ and $H^{*,*}(\Gr_3(\cU);\Z/2)$, to
demonstrate our general results.  We also make some remarks about
$H^{*,*}(\Gr_4(\cU);\Z/2)$.  

\medskip

Write $\M_2[\und{c}]\subseteq H^{*,*}(\Gr_k(\cU);\Z/2)$ for the
$\M_2$-subalgebra generated by the $c_i$'s.  We have seen that the rank chart
for the cohomology ring breaks up naturally into lines of slope
$\frac{1}{2}$, and it will be convenient to consider a corresponding
decomposition at the level of algebra.  To this end, let 
$F_i\subseteq H^{*,*}(\Gr_k(\cU);\Z/2)$ be the $\M_2$-submodule
spanned by the elements of our standard basis having degrees $(p,q)$
for
$0\leq 2q-p\leq i$.  Note that $F_0=\M_2[\und{c}]$, and in general 
$F_i$ is an $\M_2[\und{c}]$-module.
 Let $Q_i=F_i/F_{i-1}$, and call this module the
``$i$-line''.  It is a free $\M_2$-module, and the ranks correspond to
the ranks of $H^{*,*}(\Gr_k(\cU);\Z/2)$ occuring along the line of
slope $\frac{1}{2}$ that passes through $(i,i)$.   The $0$-line is
simply $\M_2[\und{c}]$.  The duality given by
Corollary~\ref{co:duality} says that the ranks along the $i$-line and
the $(k-i)$-line are the same, for every $i$.  

We take the perspective that the $0$-line is completely understood, as
this is just the polynomial ring over $\M_2$ on the classes
$c_1,c_2,\ldots,c_k$.  In some sense we then also understand the
$k$-line, by duality.  Our next observation is that we can also understand
the $1$-line (and therefore the $(k-1)$-line along with it).

\begin{lemma}
\label{le:1-line}
Let $X=H^{*,*}(\Gr_k(\cU);\Z/2)$.  Then 
\[ \rank^{2p+1,p+1}(X)=\rank^{2p,p}(X)+\rank^{2p-2,p-1}(X) + \cdots +
\rank^{2p-2(k-1),p-(k-1)}(X)
\]
for any $p\in \Z$.
\end{lemma}

\begin{proof}
We change this into a statement about partitions, using
Proposition~\ref{pr:rank-Inv}.  The claim is that
\[ \partition_{2p+1,\leq k}[1]=\sum_{i=0}^{k-1} \partition_{2p-2i,\leq k}[0].
\]
We sketch a bijective proof of this.  Regard a partition with at most
$k$ pieces as a partition having exactly $k$ pieces, but where some
pieces are $0$.  Given a partition of $2p$ into $k$ pieces that are
all even, make a partition of $2p+1$
by adding $1$ to the smallest piece.  Given a partition of $2p-1$ into
$k$ pieces that are all even, make a partition of $2p+1$ by adding $3$
to the second smallest piece.  And so on: given an element of
$\partition_{2p-2i,\leq k}[0]$, make a partition of $2p+1$ by adding
$2i+1$ to the $i$th smallest piece.  We leave it to the reader to
check that this does indeed give the desired bijection.
\end{proof}

\begin{prop}
The $1$-line $Q_1$ is a free $\M_2[\und{c}]$-module generated by the
classes $w_1^{(e)}$ for $0\leq e\leq k-1$.  
\end{prop}

\begin{proof}
We have the evident map 
\begin{myequation}
\label{eq:map} 
\M_2[\und{c}]\langle
w_1,w_1^{(1)},\ldots,w_1^{(k-1)}\rangle \ra Q_1.
\end{myequation}
Theorem~\ref{th:indecomp} says that $X$ is generated as an $\M_2$-algebra by
products of elements $c_i$ and $w_j^{(e)}$.  The only such products
that can lie on the $1$-line are products of $c_i$'s with
$w_1^{(e)}$'s.  This shows that the map in (\ref{eq:map}) is surjective.
But Lemma~\ref{le:1-line} shows that the ranks of the domain and
target of (\ref{eq:map}) coincide, hence the map must be an isomorphism.
\end{proof}

In the cohomology of $\Gr_2(\cU)$ we only have the $0$-line, $1$-line,
and $2$-line, and the outer two are dual---so we basically understand
everything.  In $\Gr_3(\cU)$ we have the $0$-line/$3$-line and the
$1$-line/$2$-line, and again we understand everything.  This is why
these two cases are fairly easy.  When we get to $\Gr_4(\cU)$ things
become more complicated.  

Let us now look in detail at $\Gr_2(\cU)$.  The rank calculations can
be done by counting partitions using Proposition~\ref{pr:rank-Inv},
and this is very easy.  One finds 
\[ \rank^{2p,p}=\rank^{2p+2,p+2}=\begin{cases} \frac{p}{2}+1 &
\text{if $p$ is even}, \\
\frac{p+1}{2} & \text{if $p$ is odd,}
\end{cases}
\]
and
\[ \rank^{2p+1,p+1}=p+1.
\]
By Theorem~\ref{th:indecomp} the indecomposables are the following elements:
\[ c_1,\ c_2,\ w_1,\ w_1^{(1)}, \ w_2.
\]
The $1$-line is a free $\M_2[\und{c}]$-module generated by $w_1$ and
$w_1^{(1)}$, and the rank calculations suggest that the $2$-line is 
the free $\M_2[\und{c}]$-module generated by $w_2$.  So we guess that
the three classes $w_1$, $w_1^{(1)}$ and $w_2$ span the cohomology as
an $\M_2[\und{c}]$-module.  If this is true, there will be relations
specifying the products of any two of the $w$-classes.  A little work
shows that
\begin{align*}
& w_1^2=\rho w_1+\tau c_1, \qquad w_2^2=\rho^2 w_2 + \rho \tau 
\bigl (w_1c_1+w_1^{(1)}\bigr ) + \tau^2 c_2\\
& \Bigl [w_1^{(1)}\Bigr ]^2=\rho\bigl (w_1^{(1)}c_1 +
w_1c_2\bigr )+\tau(c_1^3+c_1c_2)\\
\end{align*}
and also that
\begin{align*}
& w_1w_2=\rho w_2 + \tau\bigl (w_1c_1+w_1^{(1)}\bigr ) \\
& w_1w_1^{(1)}=\rho w_1^{(1)}+\tau c_1^2 +w_2c_1 \\
& w_2 w_1^{(1)}=\rho w_2 c_1 + \tau (w_1c_1^2+w_1^{(1)}c_1+w_1c_2).
\end{align*}
We have separated the relations into two classes: the relations for
the squares of the $w$-classes will always be present, but the
relations amongst square-free monomials in the $w$-classes depend very
much on the value of $k$.

Once these relations have been verified, we have a surjective algebra map
\[ \M_2[c_1,c_2,w_1,w_1^{(1)},w_2]/(R) \fib H^{*,*}(\Gr_2(\cU);\Z/2)
\]
where $R$ is the above list of relations.  As an
$\M_2[\und{c}]$-module the domain is free with generators $1$, $w_1$,
$w_1^{(1)}$, and $w_2$, and our rank calculations then show that the
Poincar\'e series for the domain and target agree.  So the above map
must be an isomorphism.

It remains to verify the relations listed above.  The ones for the
squares of $w_1$
and $w_1^{(1)}$ follow readily from (\ref{eq:w-square}) and
Lemma~\ref{le:w1e}.
For $w_2^2$ we write
\begin{align*} 
[a_1a_2]^2=[a_1^2a_2^2]=[(\rho a_1+\tau b_1)(\rho a_2+\tau
b_2)] & =\rho^2[a_1a_2] + \rho\tau [a_1b_2] + \tau^2 [b_1^2] \\
&= \rho^2 w_2
+ \rho\tau \bigl ( [a_1][b_1]+[a_1b_1]\bigr ) +\tau^2 c_1^2.
\end{align*}
Of the remaining three relations, we leave the first two to the
reader and only verify the last:
\begin{align*} [a_1a_2]\cdot [a_1b_1]= [a_1^2a_2b_1] =
\rho[a_1a_2b_1]+\tau[a_1b_2^2]
&= \rho [a_1a_2][b_1] + \tau \bigl ( [a_1][b_1^2]+[a_1b_1^2]  \bigr
)\\
&= \rho w_2c_1 + \tau \bigl ( w_1c_1^2 + w_1^{(2)}\bigr ).
\end{align*}
Now use Lemma~\ref{le:w1e} to decompose $w_1^{(2)}$.  

\vspace{0.1in}

Next let us look at the cohomology of $\Gr_3(\cU)$.  The
indecomposables are
\[ c_1, \ c_2,\ c_3,\ w_1,\ w_1^{(1)},\ w_1^{(2)},\ w_2,
\]
and the $1$-line is generated over $\M_2[\und{c}]$ by $w_1$,
$w_1^{(1)}$, and $w_1^{(2)}$.  
The evident elements of interest on the $2$-line are
\[ w_2, \ w_1\cdot w_1^{(1)}, \ w_1\cdot w_1^{(2)}.\]
Duality between the $1$-line and $2$-line suggests that we will have
three generators as an $\M_2[\und{c}]$-module, and since these are the
only candidates there is not much choice for what can happen.
Finally, we expect by duality that the $3$-line is the free
$\M_2[\und{c}]$-module generated by $w_1w_2$.  This gives a
conjectural description of the cohomology as a module over
$\M_2[\und{c}]$, which we will soon see is correct.  

The guess suggests that we should have relations for the products
$w_2\cdot w_1^{(1)}$, $w_2\cdot w_1^{(2)}$, and $w_1^{(1)}\cdot
w_1^{(2)}$---as well as for the squares of all the $w$-classes, of
course.
Some tedious work in the ring of invariants reveals the following relations:
\begin{align*}
& w_1^2=\rho w_1+\tau c_1 \\
& w_2^2=\rho^2 w_2 + \rho\tau(w_1c_1+w_1^{(1)})+\tau^2 c_2 \\
& \Bigl [ w_1^{(1)}\Bigr ]^2= \rho w_1^{(2)}+ \tau \bigl [ c_1^3 +
c_1c_2+c_3\bigr ]\\
& \Bigl [ w_1^{(2)}\Bigr ]^2= \rho \bigl [
w_1^{(2)}c_1^2 + w_1^{(1)}c_1c_2+w_1c_1c_3+w_1^{(2)}c_2+w_1^{(1)}c_3
\bigl ]\\
&\qquad\qquad\qquad\qquad\qquad\qquad + \tau\bigl [
c_1^5+c_1^3c_2+c_1^2c_3+c_1c_2^2+c_2c_3\bigr ]
\end{align*}
and
\begin{align*}
& w_2\cdot w_1^{(1)}=
w_1w_2c_1+(\rho,\tau)\\
& w_2\cdot w_1^{(2)}= 
w_1w_2c_1^2+(\rho,\tau) \\
%
& w_1^{(1)}\cdot
w_1^{(2)}=w_2c_3+w_2c_1c_2+w_1w_1^{(1)}c_1^2+w_1w_1^{(2)}c_1
+(\rho,\tau).
\end{align*}
In the last three cases  we are being somewhat lazy and not writing out the entire
relations, which are long and complicated.  We have instead written
``$(\rho,\tau)$'' as shorthand for all terms belonging to the ideal
$(\rho,\tau)$.  

Once again, we have now produced a surjective map
\[ \M_2[c_1,c_2,w_1,w_1^{(1)},w_1^{(2)},w_2]/(R) \fib H^{*,*}(\Gr_3(\cU);\Z/2)
\]
where $R$ is the set of relations above.  The domain is a free
$\M_2[\und{c}]$-module generated by
$1,w_1,w_1^{(1)},w_1^{(2)},w_2, w_1\cdot w_1^{(1)}, w_1\cdot
w_1^{(2)},w_1w_2$.  One can analyze the Poincar\'e series for the
cohomology ring in terms of partitions, and a little work shows that
the Poincar\'e series of the domain and codomain agree.  It follows
that the 
above map is an isomorphism of algebras.

\medskip
Finally, we make some brief remarks about $\Gr_4(\cU)$.  The
indecomposables are
\[ c_1,\ c_2,\ c_3,\ c_4,\  w_1,\ w_1^{(1)}, \ w_1^{(2)}, \
w_1^{(3)},\ w_2, \ w_2^{(1)}, \ w_4.
\]
The $0$-line is the polynomial algebra $\M_2[c_1,c_2,c_3,c_4]$, and
the $1$-line is the free $\M_2[\und{c}]$-module with basis elements
$w_1^{(e)}$ for $0\leq e\leq 3$.  The monomials on the $2$-line are
\[ w_2,\ w_1w_1^{(1)},\ w_1w_1^{(2)},\ w_1w_1^{(3)},\ w_2^{(1)},\
w_1^{(1)}\cdot w_1^{(2)},\ w_1^{(1)}\cdot w_1^{(3)},\ w_1^{(2)}\cdot
w_1^{(3)},
\]
with bidegrees 
\[ (2,2), \ (4,3),\  (6,4),\  (8,5),\  (6,4),\  (8,5),\ 
(10,6),\  (12,7).
\]  
The ranks along the $0$-line constitute the sequence
$S=(1,1,2,3,5,6,9,11,\ldots)$.  If the $2$-line were free on the above
generators then the ranks along the $2$-line would be
$P=(1,2,5,9,15,23,34,47,\ldots)$.  This sequence is obtained by adding
up eight copies of $S$ with appropriate shifts, according to the
topological degrees of the eight monomials listed above: $P=\sum_i
(\Sigma^{(p_i-2)/2}S)$ where $p_i$ is the topological degree of the $i$th
element of the list (we subtract two because our $2$-line ``starts''
at $w_2$).  That is,
\[ P=S+\Sigma S + \Sigma^2 S+\Sigma^2 S + \Sigma^3 S+\Sigma^3
S+\Sigma^4 S+\Sigma^5 S.
\]
Computations with partitions reveals that the actual rank sequence for
the $2$-line is $(1,2,5,8,14,20,30,40,55,\ldots)$.  Playing around
with the numerology shows that removing a $\Sigma^3S$ and the
$\Sigma^5 S$ from $P$ seems to yield the correct answer; this leads to the guess that there is a dependence relation
amongst the two elements $w_1w_1^{(3)}$ and $w_1^{(1)}w_1^{(2)}$, and
also that there should be a relation for $w_1^{(2)}w_1^{(3)}$.  One
can indeed find such relations, although the process is
time-consuming.  In the first case the relation is
\[
w_1w_1^{(3)}+w_1^{(1)}w_1^{(2)}+w_1w_1^{(2)}c_1+w_2^{(1)}c_1+w_2c_3+w_1w_1^{(1)}c_2+(\rho,\tau)=0
\]
where the last term represents an element in the ideal $(\rho,\tau)$
that we have not gone to the trouble of determining.

It again appears that the cohomology of $\Gr_4(\cU)$ is free as a
module over $\M_2[\und{c}]$, with basis consisting of certain products
of $w$-classes.  However, there does not seem to be a canonical
choice for the basis: e.g., there is no preferred choice among $w_1w_1^{(3)}$ and
$w_1^{(1)}w_1^{(2)}$ for which to include.  Also, the relations are getting truly
horrendous.  We choose to stop here.  

\section{Connections to motivic phenomena}
\label{se:motivic}

Let $F$ be a field, not of characteristic $2$.  For an algebraic
variety $X$ over $F$, a \dfn{quadratic
bundle} over $X$ is an algebraic vector bundle $E\ra X$ together with
a pairing $E\tens_F E \ra \cO_X$ that is symmetric and restricts to
nondegenerate bilinear forms on each fiber.  
For reasons that we will not explain here, such bundles play the role
in motivic homotopy theory that ordinary real vector bundles play in
classical algebraic topology (see Remark~\ref{re:quadratic=real} below
for a bit more information).  It is natural, therefore, to try to
understand characteristic classes for quadratic bundles with values in
mod $2$ motivic cohomology.  

One can make a guess at a classifying space for quadratic vector
bundles, as follows (this is known to be a true classifying
space if one works stably, by a result of \cite{ST}).  Equip the
affine space $\A^{2n}$ with the quadratic form
\[ q_{2n}(x_1,y_1,x_2,y_2,\ldots,x_n,y_n)=x_1y_1+\cdots+x_ny_n \]
and equip $\A^{2n+1}$ with the quadratic form
\[
q_{2n+1}(x_1,y_1,x_2,y_2,\ldots,x_n,y_n,z)=x_1y_1+\cdots+x_ny_n+z^2. 
\]
These are called the \dfn{split} quadratic forms.  Note that we have
$\A^{2n}$ sitting inside $\A^{2n+1}$ as the $z=0$ subspace, which
exhibits $q_{2n}$ as the restriction of $q_{2n+1}$.  We will also
regard $\A^{2n+1}$ as sitting inside $\A^{2n+2}$ as the subspace
$x_{n+1}=y_{n+1}$, which exhibits $q_{2n+1}$ as the restriction of
$q_{2n+2}$.

From now on we will write $(\A^N,q)$ for either $(\A^{2n},q_{2n})$ or
$(\A^{2n+1},q_{2n+1})$.  Note that we have a series of inclusions
\[  (\A^1,q) \inc (\A^2,q) \inc (\A^3,q)\inc \cdots
\]

Define the \dfn{orthogonal Grassmannian} $\OGr_k(\A^N)$ to be the
Zariski open subspace of $\Gr_k(\A^N)$ consisting of the $k$-planes
where $q$ restricts to a nondegenerate form.  Taking the colimit over
$N$ gives a motivic space $\OGr_k(\A^\infty)$, in the sense of
\cite{MV}.  It is an interesting (and unsolved) problem to compute the
motivic cohomology groups of this space.

Now restrict to the case $F=\R$.  From an $\R$-variety $X$ we can
consider the set $X(\C)$ of $\C$-valued points, regarded as a
topological space via the analytic
topology.  This space
has an evident $\Z/2$-action given by complex conjugation, and the
assignment $X\mapsto X(\C)$ extends to a map of homotopy theories from
motivic homotopy theory over $\R$ to $\Z/2$-equivariant homotopy
theory.  Our goal in this section is only to note the following
result:

\begin{thm}
\label{th:Z/2}
There is an equivariant weak homotopy equivalence
\[ [\OGr_k(\A^{N})](\C) \
\he \Gr_k(\cU^N).
\]
(Recall that $\cU^N$ denotes the first
$N$ summands of the infinite $\Z/2$-representation $\cU=\R\oplus \R_-\oplus \R
\oplus \R_- \oplus \cdots$).
\end{thm}

The above theorem shows that the main problem considered in this paper
is indeed the $\Z/2$-equivariant analog of the problem of motivic
characteristic classes for quadratic bundles.

We will need a few preliminary results before giving the proof of the
theorem.  To generalize our previous definition somewhat, if $V$ is
any vector space with a quadratic form $q$ then we write $\OGr_k(V)$
for the subspace of $\Gr_k(V)$ consisting of $k$-planes $W\subseteq V$
such that $q|_W$ is nondegenerate.  Sometimes $V$ will be a real
vector space and sometimes $V$ will be a complex vector space, and in
the latter case our orthogonal Grassmannian will be the space of
complex $k$-planes on which $q$ is nondegenerate.  Usually the intent
will be clear from context.

Assume $V$ is real and the form $q$ is positive-definite.  This form
extends to give a complex quadratic form on $V\tens_\R \C$ that we
will also call $q$.  The
complexification map $c\colon \Gr_k(V) \ra \Gr_k(V\tens_\R \C)$ has
its image contained in $\OGr_k(V\tens_\R \C)$.  To see this, just
observe that if $U\subseteq V$ is any $k$-plane then there is a basis
for $U$ with respect to which $q$ looks like the sum-of-squares form.
Extending this basis to $U\tens_\R \C$ shows that $q$ is nondegenerate
here.  Similar remarks apply to show that the direct-sum map in part
(b) of the following result takes its image in $\OGr$ rather than just
$\Gr$.

Note that the following result takes place in the non-equivariant
setting:

\begin{prop} 
\label{pr:quad-form-nonequiv}
Let $V$ be a real vector space with a positive-definite
quadratic form $q$.  
\begin{enumerate}[(a)]
\item The complexification map $\Gr_k(V) \ra \OGr_k^\C(V\tens_\R \C)$ is a weak
homotopy equivalence;
\item Let $V'$ be another real vector space with positive-definite
form $q'$.  Then the direct-sum map $\coprod_{a+b=k} \Gr_a(V)\times \Gr_b(V') \ra
\OGr_k(V\oplus V',q\oplus (-q'))$ is a weak homotopy equivalence.
\end{enumerate}
\end{prop}

\begin{proof}
Without loss of generality we may assume that $V=\R^n$ and $q$ is the
sum-of-squares form.  Recall that the symmetry group of this form is
the Lie group
 $O_n=\{A\in M_{n\times n}(\R)\,|\,
AA^T=I\}$. 
The symmetry group for the sum-of-squares form over $\C$
is
$O_n(\C)=\{A\in M_{n\times n}(\C)\,|\,
AA^T=I\}$.  Recall that
$O_n$ is a maximal compact subgroup inside of $O_n(\C)$; it is
therefore known by the
Iwasawa decomposition that the inclusion $O_n\inc O_n(\C)$  is a
homotopy equivalence (see \cite[Theorem 8.1 of Segal's lecture]{CSM}
or \cite[Chapter XV, Theorem 3.1]{H}).

The space $\Gr_k(\R^n)$ is homeomorphic to $O_n/[O_k\times O_{n-k}]$.
Likewise, $\OGr_k(\C^n)$ is homeomorphic to $O_n(\C)/[O_k(\C)\times
O_{n-k}(\C)]$.  The map in part (a) is the evident comparison map
between these homogeneous spaces.  Consider the two fiber bundles
\[ \xymatrix{
O_k\times O_{n-k} \ar[r] \ar[d] & O_n \ar[r]\ar[d] & O_n/[O_k\times
O_{n-k}]\ar[d] \\
O_k(\C)\times O_{n-k}(\C) \ar[r] & O_n(\C) \ar[r] & O_n(\C)/[O_k(\C)\times
O_{n-k}(\C)]
}
\]
(written horizontally).
The left and middle vertical maps are weak equivalences, therefore the
right map is as well.  This proves (a).

For (b) recall that a nondegenerate quadratic form on an $n$-dimensional real vector
space is classified by its signature: the pair of integers $(a,b)$ such that $a+b=n$,
representing the number of positive and negative entries in any
diagonalization of the form.  Let $O(a,b)$ be the symmetry group for
the quadratic form of signature $(a,b)$.  This Lie group contains
$O(a)\times O(b)$ in the evident way, and it is known that this is a
maximal compact subgroup.  Consequently, the inclusion $O(a)\times
O(b)\inc O(a,b)$ is a weak homotopy equivalence by the Iwasawa decomposition.  

We can assume $V=\R^n$ and $V'=\R^{n'}$, with both $q$ and $q'$ being
the sum-of-squares form.  The group $O(n,n')$ acts on $\OGr_k(V\oplus
V')$ in the evident way.  It is easy to see that the action decomposes the
orthogonal Grassmannian into a disjoint union of orbits, one for every
possible signature $(a,b)$ with $a+b=k$.  The path component
corresponding to such a signature is the homogeneous space
\[ O(n,n')/ [O(a,b)\times O(n-a,n'-b)].
\]

The map in part (b) coincides with the disjoint union of the evident maps
\[ \xymatrix{ 
\Bigl [ O(n)/[O(a)\times O(n-a)] \Bigr ]  \times \Bigl [ 
O(n')/[O(b)\times O(n'-b)]\Bigr ] \ar[d]^\iso \\
[O(n)\times O(n')]/\Bigl [[O(a)\times O(n-a)] \times [O(b)\times
O(n'-b)]\Bigr ] \ar[d]
\\ O(n,n')/ [O(a,b)\times O(n-a,n'-b)].
}
\]
At this point one proceeds exactly in the proof of part (a): write
down a map between two fiber bundles, where two of the three
maps are already known to be weak homotopy equivalences.
\end{proof}

We next move into the equivariant setting.
By an \dfn{orthogonal representation} of $\Z/2$ we mean
a pair $(V,q)$ where $V$ is a real vector space and $q\colon
 V\ra \R$ is a positive-definite quadratic form on $V$ such that
$q(\sigma x)=q(x)$ for all $x\in V$.  The main examples for
us will be where $V=\R^n$, $q$ is the standard sum-of-squares
form, and $\Z/2$ acts on $V$ by changing signs on some subset of the
standard basis elements.  

Let $V_\C=V\tens_\R \C$, with the $\Z/2$ action induced by that on
$V$. 
The
complexification  map $\Gr_k(V)\ra \OGr_k(V_\C)$ sending
$U\subseteq V$ to $U_\C \subseteq V_\C$ is
 clearly equivariant, where the $\Z/2$-actions on domain and
codomain are induced by those on $V$ and $V_\C$.

\begin{cor}
\label{co:Gr=OGr}
For any orthogonal representation $V$ of $\Z/2$, the map of
$\Z/2$-spaces $\Gr_k(V)\ra
\OGr_k(V_\C)$ is an equivariant weak equivalence.
\end{cor}

\begin{proof}
Taking Proposition~\ref{pr:quad-form-nonequiv}(a) under consideration, it
suffices to prove that the induced map of fixed sets is a weak
equivalence.  Let $V^{\Z/2}$ and $V^{-\Z/2}$ denote the $+1$ and $-1$
eigenspaces for the involution on $V$.  These are orthogonal with
respect to the
inner product on $V$.  A subspace $U\subseteq V$ is fixed under the $\Z/2$
action
if and only if $U$ equals the direct sum $(U\cap V^{\Z/2})\oplus
(U\cap V^{-\Z/2})$.  
From this we get a homeomorphism
\[ \Gr_k(V)^{\Z/2} \iso 
\coprod_i \Gr_i(V^{\Z/2})\times
\Gr_{k-i}(V^{-\Z/2}),
\]
which sends $U\subseteq V$ to the pair $(U\cap V^{\Z/2}, U\cap
V^{-\Z/2})$.  In the same way, one obtains a homeomorphism
\[ \OGr_k(V_\C)^{\Z/2}\iso
\coprod_i \OGr_i(V_\C^{\Z/2})\times
\OGr_{k-i}(V_\C^{-\Z/2}).
\]
Since the inclusions $\Gr_i(V^{\Z/2})\inc \OGr_i(V_\C^{\Z/2})$ and
$\Gr_{j}(V^{-\Z/2})\inc \OGr_{j}(V_\C^{-\Z/2})$ are (non-equivariant) weak
equivalences by Proposition~\ref{pr:quad-form-nonequiv}(a), this completes the proof.
\end{proof}

The above corollary has been included for completeness, but it
actually does not give us what we need.  The $\Z/2$ action on
$V\tens_\R \C$ is complex linear, whereas we will find that we
actually need to consider conjugate linear actions.  We do this next.

Let $W$ be a complex vector space with a nondegenerate quadratic form
$q$.  Let $\sigma\colon W\ra W$ be a
conjugate-linear map such that $\sigma^2=1$.  That is,
$\sigma(zx)=\bar{z}\sigma(x)$ for every $z\in \C$ and $x\in W$.  Also
assume that $q(\sigma x)=\overline{q(x)}$ for every $x\in W$.  The
space $\OGr_k(W)$ then has a $\Z/2$-action induced by $\sigma$: if
$J\subseteq W$ is a complex subspace such that $q|_J$ is
nondegenerate, then $\sigma(J)$ is another complex subspace on which
$q$ restricts to be nondegenerate.  Our next task is to analyze the
fixed space $\OGr_k(W)^{\Z/2}$.  

\begin{remark}
Let $(V,q)$ be an orthogonal representation for $\Z/2$, and let $W$ be the
vector space $V\tens_\R \C$ with the action given by $\sigma(v\tens
z)=\sigma(v)\tens \bar{z}$.  Then $(W,q)$ satisfies the conditions of
the above paragraph.  In this case we will use the notation
$W=V\tens_\R \overline{\C}$.  The bar over the $\C$ just reminds us
that $\Z/2$ acts on that factor by conjugation. 
\end{remark}

Returning to the case of a general $W$, 
note that as a real vector space $W$ decomposes as $W^{\Z/2}\oplus
W^{-\Z/2}$, where the summands are the subspaces on which $\sigma$
acts as the identity and as multiplication by $-1$.  Moreover,
multiplication by $i$ maps $W^{\Z/2}$ isomorphically onto
$W^{-\Z/2}$.  Finally, one easily checks that $q$ is real-valued on
both $W^{\Z/2}$ and
$W^{-\Z/2}$.

If $J\subseteq W$ is any complex subspace that is fixed
by $\sigma$ then we have the decomposition $J=(J\cap W^{\Z/2})\oplus (J\cap W^{-\Z/2})$,
and multiplication by $i$ interchanges the two summands.  In this way
we get a map
\[ \OGr_k(W,q)^{\Z/2}\lra \Gr_k(W^{\Z/2}), \qquad J\mapsto J\cap W^{\Z/2}\]
and the image is readily checked to land in $\OGr_k(W^{\Z/2},q)$.  
Conversely, if $M\subseteq W^{\Z/2}$ is any $k$-dimensional real
subspace such that $q|_M$ is nondegenerate then $M\oplus iM\subseteq
W$ is a $k$-dimensional complex subspace with the same property.  So
we also get a map $\OGr_k(W^{\Z/2}) \ra \OGr_k(W,q)^{\Z/2}$.  It is routine
to check that these maps are inverse isomorphisms.  Thus, we have
proven the following:

\begin{prop}
In the above setting, there is a homeomorphism $\OGr_k(W,q)^{\Z/2}\iso
\OGr_k(W^{\Z/2},q)$.  
\end{prop}

We are now ready to prove the main theorem of this section:

\begin{proof}[Proof of Theorem~\ref{th:Z/2}]
Write $q_{sp}$ for the split quadratic form on $\C^N$, and $q_{ss}$
for the sum-of-squares quadratic form on $\C^N$.  
The theorem concerns
the space $\OGr_k(\C^N,q_{sp})$ where the $\Z/2$-action is
induced by complex conjugation.  Let $x_1,y_1,x_2,y_2,\ldots $ denote our
standard coordinates on $\C^N$, with the convention that when $N$ is
odd then the last of the $y_j$'s is just zero.
By changing coordinates we can change $q_{sp}$ into $q_{ss}$.
Precisely, define a map $\phi\colon \C^N\ra \C^N$ by
\[
\phi(x_1,y_1,x_2,y_2,\ldots)=(x_1+iy_1,x_1-iy_1,x_2+iy_2,x_2-iy_2,\ldots).
\]
Then we have
$q_{sp}(\phi(v))=q_{ss}(v)$ for any $v\in \C^N$.  This gives us an
identification of non-equivariant spaces $\OGr_k(\C^N,q_{sp})\iso
\OGr_k(\C^N,q_{ss})$.  To extend this to an equivariant identification,
note that if the target of $\phi$ is given the conjugation action then
the domain of $\phi$ gets the action that both conjugates all
coordinates AND changes the signs of the
$y$-coordinates.  In terms of previously-established notation, this is the equivariant
homeomorphism
\[ \OGr_k(\C^N,q_{sp}) \iso \OGr_k(\cU^N\tens \overline{\C},q_{ss}).\]

Consider the complexification map
\[ c\colon \Gr_k(\cU^N) \ra \OGr_k(\cU^N\tens \overline{\C},q_{ss}).
\]
We have seen in Proposition~\ref{pr:quad-form-nonequiv}(a) that this
is a non-equivariant weak equivalence.  To analyze what is happening
on fixed sets, let $W=\cU^N\tens \overline{\C}$.  Note that
$W^{\Z/2}=\{(r_1,ir_2,r_3,ir_4,\ldots,(i)r_N)\,|\, r_1,\ldots,r_N\in
\R\}$, where the last coordinate has the $i$ in front when $N$ is even.
Note as well that we can decompose $W^{\Z/2}=W^{\Z/2}_+\oplus W^{\Z/2}_-$
where
\[ W^{\Z/2}_+=\{(r_1,0,r_3,0,\ldots)\,|\, r_i\in \R\}, \qquad
W^{\Z/2}_-=\{(0,ir_2,0,ir_4,\ldots)\,|\, r_i\in \R\}.
\]
The form $q_{ss}$ is positive definite on the first summand and
negative definite on the second.

Let $\cU^N_+$ and $\cU^N_-$ be the subspaces spanned by the odd- and
even-numbered basis elements, respectively.  So $\cU^N_+=(\cU^N)^{\Z/2}$ and
$\cU^N_-=(\cU^N)^{-\Z/2}$.  
Note the following maps:
\[ \xymatrixcolsep{0.001pc}\xymatrix{
\Gr_k(\cU^N)^{\Z/2} \ar[r]^-c &
\OGr_k(W,q_{ss})^{\Z/2} &
\OGr_k(W^{\Z/2},q_{ss})\ar[l]_-\iso \\
\coprod\limits_{a+b=k} \Gr_a(\cU^N_+)\times \Gr_b(\cU^N_-) \ar@{=}[u]\ar@{.>}@<0.5ex>[rr] &&
\coprod\limits_{a+b=k} \Gr_a(W^{\Z/2}_+) \times \Gr_b(W^{\Z/2}_-) \ar[u]_\sim
}
\]
The map on the right is the evident one, and is a weak homotopy
equivalence by  Proposition~\ref{pr:quad-form-nonequiv}(b).  The
dotted map is the obvious homeomorphism, obtained by identifying
$\cU^N_+=W^{\Z/2}_+$, $i\cdot\cU^N_-= W^{\Z/2}_-$.
One readily checks that the diagram commutes, and this verifies that
$c$ induces a weak homotopy equivalence of fixed sets.  Thus, $c$ is
an equivariant weak equivalence.
\end{proof}

\begin{remark}
\label{re:quadratic=real}
The non-equivariant part of Theorem~\ref{th:Z/2} (equivalently,
Proposition~\ref{pr:quad-form-nonequiv}(a)) gives the homotopy
equivalence of spaces $\OGr_k(\C^N) \he
\Gr_k(\R^N)$.  This is a classical result: for example, see \cite[remarks in
Section 1.5]{A1} and \cite[discussion of real Grassmannians throughout
Chapter 5]{S}.  Notice that this gives some corroboration to the idea
that quadratic bundles are the motivic analogs of real vector
bundles.  
\end{remark}


\appendix

\section{The deRham ring of invariants in characteristic two}

Let $K_n=\Lambda(a_1,\ldots,a_n)\tens \F_2[b_1,\ldots,b_n]$, and let $\Sigma_n$
act on $K_n$ by simultaneous permutation of the $a_i$'s and $b_j$'s.
Let $L_n=K_n^{\Sigma_n}$.  We call $L_n$ the ``deRham ring of
invariants''.  Note that there is an augmentation $\epsilon\colon K_n
\ra \F_2$ sending all the $a_i$'s and $b_j$'s to zero, and this
restricts to an augmentation of $L_n$.  Let $I\subseteq L_n$ be the
augmentation ideal.  Our first aim in this section is to give a vector
space basis for the module of indecomposables $I/I^2$.  Said
differently, we give a minimal set of generators for the ring $L_n$.  

Note that $K_{n+1}$ maps to $K_n$ by sending $a_{n+1}$ and $b_{n+1}$
to zero, and this homomorphism induces an algebra map $L_{n+1}\ra
L_n$.  That is, if $f(a,b)$ is a polynomial expression in the $a$'s
and $b$'s that is invariant under the $\Sigma_{n+1}$-action, then
eliminating all monomials with an $a_{n+1}$ or $b_{n+1}$ produces a
polynomial that is invariant under $\Sigma_n$.  From this description
it is also clear that $L_{n+1}\ra L_n$ is surjective: if $f(a,b)$ is a
$\Sigma_n$-invariant then one can make a $\Sigma_{n+1}$-invariant by 
 adding on appropriate monomial terms that all have $a_{n+1}$ or $b_{n+1}$.

Let $L_\infty$ be the inverse limit of the system
\[ \cdots \lra  L_3 \lra L_2 \lra L_1.
\]
The second goal of this section is to give a complete description of
the ring $L_\infty$.

These results are presumably well-known amongst algebraists.  See
Section 7 of \cite{R} for the case of
$\F_2[a_1,\ldots,a_k,b_1,\ldots,b_k]$, which can be used to deduce
some of our
results.  See also \cite[Section 2]{GSS} for some related work.  Rather
than use the machinery of \cite{R}, however, we have chosen to give a
`low-tech' treatment which is perhaps more illuminating for our
present purposes.

\medskip

If $m\in K_n$ is a monomial in the $a_i$'s and $b_j$'s, write $[m]$ for the
smallest polynomial that contains $m$ as one of its terms and is
invariant under the $\Sigma_n$-action.  Here `smallest' is measured in
terms of the number of monomial summands.  We can also describe $[m]$
as
\[ [m]=\sum_{\sigma\in \Sigma_n/H} \sigma.m \]
where $H$ is the stabilizer of $m$ in $\Sigma_n$.  

Using the above noation, write $\alpha_{i,e}=[a_1\ldots
a_{2^i}b_1^e\ldots b_{2^i}^e]$ for $1\leq 2^i\leq n$ and $0\leq e$.
Also, write $\sigma_i(a)$ and $\sigma_i(b)$ for the elementary
symmetric functions in the $a$'s and $b$'s, respectively.  So
$\sigma_i(a)=[a_1\ldots a_i]$, for example.

We can now state the main result:

\begin{thm}\label{th:En}
\mbox{}\par
\begin{enumerate}[(a)]
\item $L_n$ is minimally generated by the classes $\sigma_i(b)$ for
$1\leq i \leq n$ together with the classes $\alpha_{i,e}$ for $1\leq
2^i \leq n$ and $0\leq e\leq
 \frac{n}{2^i} -1$.  That is to say, these classes give
a vector space basis for $I/I^2$.  
\item The number of indecomposables for $L_n$ is 
\[ 3n-(\text{$\#$ of
ones in the binary expansion for $n$}).
\]
\item $L_\infty \iso \Lambda \bigl (\alpha_{i,e}\,|\, 0\leq i, 0\leq
e\bigr )\tens
\F_2[\sigma_1(b),\sigma_2(b),\ldots]$.
\end{enumerate}
\end{thm}

The Online Encyclopedia of Integer Sequences \cite{OE} was useful in discovering
the formula in part (b).

The proof of this theorem will be given after establishing several
lemmas.  The first result we give is not directly needed for the
proof, but is   included for two reasons: it provides some context
that helps explain the more complicated theorem above, and we actually
need the result in the proof of Proposition~\ref{pr:w-indecomposable}.  The result
is probably well-known, but we are not aware of a reference.  

\begin{prop} 
\label{pr:exterior-inv}
Let $\Sigma_n$ act on $\Lambda_{\F_2}(a_1,\ldots,a_n)$ by
permutation of indices.  Then
\[
\Lambda(a_1,\ldots,a_n)^{\Sigma_n}=
\Lambda(\sigma_1,\sigma_2,\sigma_4,\ldots,\sigma_{2^k})/R
\]
where $k$ is the largest integer such that $2^k\leq n$ and $R$
is the ideal generated by all products
$\sigma_{2^{i_1}}\sigma_{2^{i_2}}\cdots\sigma_{2^{i_s}}$
where $2^{i_1}+2^{i_2}+\cdots+2^{i_s}>n$.
\end{prop}

\begin{proof}
It is easy to see that the classes $1,\sigma_1,\ldots,\sigma_n$ form a
vector space basis for the ring of invariants over $\F_2$.  
Put a grading on $\Lambda(a_1,\ldots,a_n)$ by having the degree of
each $a_i$ be $1$.  Then the ring of invariants is also graded;
the dimension of each homogeneous piece equals $1$ in degrees
from $0$ through $n$, and zero in degrees larger than $n$.

It is also easy to see that $\sigma_i^2=0$ for each $i$, and so we get
a map of rings
\[ \Lambda(\sigma_1,\ldots,\sigma_n)/R \fib
\Lambda(a_1,\ldots,a_n)^{\Sigma_n}.
\]

The next thing to note is that $\sigma_r\cdot
\sigma_s=\tbinom{r+s}{r}\sigma_{r+s}$.  This is an easy computation:
distributing the product in $[a_1\ldots a_r]\cdot [a_1\ldots a_s]$ one
finds that the products of monomials are all zero if the monomials
have any variables in common.  The products that are not zero have the
form $a_{i_1}\ldots a_{i_{r+s}}$, and such a monomial appears exactly
$\binom{r+s}{r}$ times.  

If $r$ is not a power of $2$ then there exists an $i$ such that
$\binom{r}{i}$ is odd, which implies that $\sigma_r=\sigma_i\cdot
\sigma_{r-i}$.  So such classes are decomposable.  We therefore have a
map
\[ \Lambda(\sigma_1,\sigma_2,\sigma_4,\ldots,\sigma_{2^k})/R \fib
\Lambda(a_1,\ldots,a_n)^{\Sigma_n}.
\]
This is a map of graded algebras, and 
the Poincar\'e Series for the domain and target are readily checked to
coincide.  Since the map is a surjection, it must be an isomorphism.
\end{proof}

We next establish a series of lemmas directly dealing with the
situation of Theorem~\ref{th:En}.
We begin by introducing some notation and terminology.  If
$I=\{i_1,\ldots,i_k\}$ then write $a_I$ for $a_{i_1}a_{i_2}\cdots
a_{i_k}$.  Likewise, if $d_I$ is a function $I\ra \Z_{\geq 0}$ then
write $b_I^{d_I}$ for the monomial
$b_{i_1}^{d_{i_1}}b_{i_2}^{d_{i_2}}\cdots b_{i_k}^{d_{i_k}}$.  If $m$
is a monomial in the $a$'s and $b$'s, then the variables $a_i$ and
$b_i$ are said to be \dfn{bound} in $m$ if $a_ib_i$ divides $m$.  If
$a_i$ divides $m$ but $b_i$ does not, we will say that $a_i$ is
\dfn{free} in $m$ (and in the opposite situation we'll say that $b_i$
is free).  Any monomial may be written uniquely in the form
\[ m=a_Ib_I^{d_I} a_J b_K^{e_K}
\]
where the indices in $I$ represent all the bound variables: so $I\cap
J = I\cap K = J\cap K =\emptyset$.  Finally, recall that $[m]$ denotes
the smallest invariant polynomial containing $m$ as one of its terms.

\begin{lemma} 
\label{le:in1}
Let $m=a_Ib_I^{d_I} a_J b_K^{e_K}$.  Then $[m]$ is
decomposable in $L_n$ if any of the following conditions are satisfied:
\begin{enumerate}[(1)]
\item  $I\neq \emptyset$ and $J\neq \emptyset$ (i.e., some of the
$a$'s are bound and some are free).
\item $J=\emptyset$ and $d_{i_1}\neq d_{i_2}$ for some $i_1,i_2\in I$.
\item $J=K=\emptyset$ and $\#I$ is not a power of $2$.
\item $I=K=\emptyset$ and $\#J$ is not a power of $2$.  
\end{enumerate}
\end{lemma}

\begin{proof}
For (1) first assume that $K=\emptyset$, and consider the product
$[a_Ib_I^{d_I}]\cdot [a_J]$.  Distributing this into sums of products
of monomials, such products vanish if $I$ and $J$ intersect.  A
typical term that remains is $a_Ia_Jb_I^{d_I}$, and it is clear that
this term occurs exactly once.  In other words,
\[ [a_Ib_I^{d_I}]\cdot [a_J]=[a_Ia_Jb_I^{d^I}].
\]

To finish the proof of (1) we do an induction on the size of $\#K$.
If $m=a_Ib_I^{d_I}a_Jb_K^{e_K}$ then consider the product
$[a_Ib_I^{d_I}]\cdot [a_Jb_K^{e_K}]$.
Distributing this into sums of products of monomials, we find that
\[ [a_Ib_I^{d_I}]\cdot [a_Jb_K^{e_K}] = [m] + \Bigl (\text{terms of the form
$[a_Ib_I^{d'_I}a_Jb_{K'}^{e_{K'}}]$ where $\#K'<\#K$}\Bigr ). \]
The latter terms come from products where the
indices in $I$ match some of those in $K$.  By induction these latter
terms are all decomposable in $L_n$, so $[m]$ is also decomposable.

For (2), we again first assume that $K=\emptyset$ so that we are looking at
$[a_1\ldots a_s b_1^{d_1}\cdots b_s^{d_s}]$.  By rearranging the
labels we may assume $d_1 \geq d_2 \geq \cdots \geq d_s$.  Let $r$ be
the smallest index for which $d_r=d_s$, and consider the product
\[ [a_1\ldots a_{r-1}b_1^{d_1}\ldots b_{r-1}^{d_{r-1}}]\cdot
[a_r\ldots a_s b_r^f\cdots b_s^f] \]
where $f=d_s$.  Once again considering the pairwise product of
monomials, all such terms vanish except for ones of the form
$a_{i_1}\ldots a_{i_s} b_{i_1}^{d_1}\ldots
b_{i_{r-1}}^{d_{r-1}}b_{i_r}^f\ldots b_{i_s}^f$.  The fact that $f$ is
the smallest degree on the $b_i$'s guarantees that this term appears
exactly once in the sum, and hence
\[ [a_1\ldots a_{r-1}b_1^{d_1}\ldots b_{r-1}^{d_{r-1}}]\cdot
[a_r\ldots a_s b_r^f\cdots b_s^f] = [a_1\ldots a_s b_1^{d_1}\ldots
b_s^{d_s}].
\]

To complete the proof of (2) we perform an induction on $\#K$.
Consider a monomial 
\[ m=a_Ib_I^{d_I}b_K^{e_K}=a_1\ldots a_sb_1^{d_1}\ldots b_s^{d_s}
b_{s+1}^{e_1}\ldots b_{s+k}^{e_k}.
\]  
Again arrange things so that
$d_1\geq d_2 \geq \cdots \geq d_s$ and let $r$ be the smallest index
for which $d_r= d_s$.  If we again write $f=d_s$, then one readily checks that
\[ [a_1\ldots a_{r-1}b_1^{d_1}\ldots b_{r-1}^{d_{r-1}}]\cdot
[a_r\ldots a_s b_r^f\cdots b_s^f b_{s+1}^{e_1}\ldots b_{s+k}^{e_k}] =
[m] + \sum [a_I b_I^{d'_I}b_{K'}^{e_{K'}}]
\]
where for each term in the sum
$K'$ is a
proper subset of $K$.  These terms inside the sum correspond to pairs
of 
monomials in the product for which a
$b_i$ for $1\leq i\leq r-1$ matches a $b_{s+j}$ for $1\leq j\leq k$.
However, by induction on $\#K$ each $[a_Ib_I^{d'_I}b_{K'}^{e_{K'}}]$ is
decomposable, hence $[m]$ is also decomposable.  

To prove (3) it suffices (in light of (2)) to show that $[a_1\ldots a_k b_1^e\ldots
b_k^e]$ is decomposable whenever $k$ is not a power of $2$.  This
assumption guarantees that $\binom{k}{i}$ is odd for some $i$ in the
range $1\leq i \leq k-1$.  We claim that
\[ [a_1\ldots a_ib_1^e\ldots b_i^e]\cdot [a_{i+1}\ldots
a_kb_{i+1}^e\ldots b_k^e]=[a_1\ldots a_kb_1^e\ldots b_k^e].
\]
To see this, note that all terms in the product vanish except for ones
of the form $a_{i_1}\ldots
a_{i_k}b_{i_1}^e\ldots b_{i_k}^e$, and such a term appears exactly
$\binom{k}{i}$ times.  Use that  $\binom{k}{i}$ is odd.

The proof of (4) is the same as for (3), it is really the special case
$e=0$.  
\end{proof}

\begin{lemma}
\label{le:in2}
$L_n$ is generated as an algebra by the elements $\sigma_i(b)$
for $1\leq i \leq n$ together with the classes $[m]$ where
$m=a_Ib_I^{d_I}a_J$ (that is, where $m$ has no free $b$'s).  
\end{lemma}

\begin{proof}
Let $Q\subseteq L_n$ denote the subalgebra generated by the elements from
the statement of the lemma.  We will prove that if
$m=a_Ib_I^{d_I}a_Jb_K^{e_K}$
is an arbitrary monomial then $[m]$ is equivalent modulo decomposables
to an element of $Q$.  This readily yields the result by an induction
on degree.     

First consider the case where $I=J=\emptyset$, so that $m=b_K^{e_K}$.
Note that $\Z/2[b_1,\ldots,b_n]\subseteq L_n$, and we know
$\Z/2[b_1,\ldots,b_n]^{\Sigma_n}$ is a polynomial algebra on the
$\sigma_i(b)$ for $1\leq i\leq n$. 
It follows at once that
$[m]$ is equivalent modulo decomposables to a multiple of a
$\sigma_i(b)$.

The next stage of the proof is done by induction on $\#K$.  The base
case $K=\emptyset$ is trivial, as such monomials lie in $Q$ by
definition.  So assume $K\neq \emptyset$ and consider the product
$[a_Ib_I^{d_I}a_J]\cdot [b_K^{e_K}]$.  This product decomposes into
a sum $[m]+[m_1]+[m_2]+\cdots$ where each $m_i$ has fewer free $b$'s
than $m$.  Therefore $[m]$ is equivalent to $\sum_i [m_i]$ modulo
decomposables, and each $[m_i]$ is equivalent to an element of $Q$ by
induction.  
\end{proof}

\begin{cor}
\label{co:Ln-gen}
$L_n$ is generated as an algebra by the following elements:
\begin{enumerate}[(1)]
\item $\sigma_i(b)$ for $1\leq i\leq n$;
\item $[a_1\ldots a_{2^i}b_1^e\ldots b_{2^i}^e]$ for $1\leq 2^i\leq n$
and $e\geq 0$.
\end{enumerate}
\end{cor}

\begin{proof}
Lemma~\ref{le:in2} gives the generators $\sigma_i(b)$ and
$[a_Ib_I^{d_I}a_J]$.  Using Lemma~\ref{le:in1}(1) we reduce the
second class  to all elements $[a_Ib_I^{d_I}]$ and $[a_J]$.  Finally,
Lemma~\ref{le:in1}(2,3,4) further reduces the class to the set of
elements in the statement of the corollary.
\end{proof}

We need one more lemma before completing the proof of
Theorem~\ref{th:En}.  For $x,y\in L_n$ let us write $x\equiv y$ to
mean $x$ and $y$ are equivalent modulo decomposables (that is, $x-y\in
I^2$).

\begin{lemma}
\label{le:in3}
If $r\geq k$ and $n\geq r+k$ then 
\[ [a_1\ldots a_k b_1^e \ldots b_k^e b_{k+1}\ldots b_{k+r}] \equiv 
[a_1\ldots a_k b_1^{e+1} \ldots b_k^{e+1} b_{k+1}\ldots b_{r}]. 
\]
Consequently, provided $ke\leq n$ one has that
\[ [a_1\ldots a_k b_1^e\ldots b_k^e]\equiv [a_1\ldots a_kb_1\ldots
b_{ke}]. 
\]
If $k+ke \leq n$ we also have
\[ [a_1\ldots a_k b_1^e\ldots b_k^e]\equiv [a_1\ldots a_kb_{k+1}\ldots
b_{k+ke}]. 
\]
\end{lemma}

\begin{proof}
For the first statement consider the product
\[ [a_1\ldots a_kb_1^e\ldots b_k^e]\cdot [b_{1}\ldots b_r].
\]
The product contains $[a_1\ldots a_k b_1^e\ldots b_k^e b_{k+1}\ldots
b_{k+r}]$ and $[a_1\ldots a_k b_1^{e+1}\ldots b_k^{e+1} b_{k+1}\ldots
b_r]$, as well as other terms that look like $[a_1\ldots a_k
b_1^{d_1}\ldots b_k^{d_k}b_{k+1}\ldots b_{k+i}]$ in which the $d_i$'s
are not all equal.  But such terms are all decomposable by
Lemma~\ref{le:in1}(2).  

The second statement follows from the first using an induction:
\begin{align*}
 [a_1\ldots a_kb_1\ldots b_{ke}] & \equiv 
[a_1\ldots a_kb_1^2\ldots b_k^2
b_{k+1}\ldots b_{ke-k}] \\
& \equiv [a_1\ldots a_kb_1^3\ldots b_k^3
b_{k+1}\ldots b_{ke-2k}] \\
& \equiv \ldots \\
&\equiv [a_1\ldots a_kb_1^e\ldots b_k^e].
\end{align*}

Finally, for the third statement we consider the product
\[ [a_1\ldots a_k]\cdot [b_1\cdots b_{ke}].\]
This is a sum of terms $[m_i]$ where $[a_1\ldots a_kb_1\ldots b_{ke}]$
appears exactly once, $[a_1\ldots a_kb_{k+1}\ldots b_{k+ke}]$ appears
exactly once, and all other $m_i$'s have at least one free $a$ and one
bound $a$.  But Lemma~\ref{le:in1}(1) then tells us that these other
$m_i$'s are all decomposable.   
\end{proof}

\begin{cor}
\label{co:Ln-gen2}
If $e> \frac{n}{k}-1$ then $[a_1\ldots a_kb_1^e\ldots b_k^e]$ is
decomposable in $L_n$.  
\end{cor}

\begin{proof}
Let $N=ke+k$, which is larger than $n$ by assumption.  We begin by
considering the element $[a_1\ldots a_kb_1^e\ldots b_k^e]$ in $L_N$.  
Lemma~\ref{le:in3} gives that
\[ [a_1\ldots a_kb_1^e\ldots b_k^e]\equiv [a_1\ldots a_kb_{k+1}\ldots
b_{k+ke}].
\]
Now apply the homomorphism $L_N \ra L_n$, and note that since $ke+k>n$
the element on the right maps to zero (every monomial term has at least one
index that is larger than $n$).  This proves that $[a_1\ldots
a_kb_1^e\ldots b_k^e]$ is decomposable in $L_n$.
\end{proof}

At this point we have verified that $L_n$ is generated, as an algebra,
by the classes $\sigma_i(b)$ for $1\leq i\leq n$ together with the
classes $[a_1\ldots a_{2^i}b_1^e\ldots b_{2^i}^e]$ for $1\leq 2^i\leq
n$ and $0\leq e \leq \frac{n}{2^i}-1$.  It remains to verify that
these classes are a minimal set of algebra generators---or
equivalently, that they give a $\Z/2$-basis for $I/I^2$.
The approach will be to first grade the algebras in a convenient way.
Then we identify the indecomposables in
$L_\infty$, which can be done by a counting argument.  Finally, we observe
that $L_\infty \ra L_n$ is an isomorphism in degrees less than or
equal to $n$, and use this to deduce the desired facts about the
indecomposables in $L_n$.  

Grade the algebra
$K_n=\Lambda(a_1,\ldots,a_n)\tens \F_2[b_1,\ldots,b_n]$ by having the
degree of each $a_i$ be $1$ and the degree of each $b_i$ be $2$.  Then
$L_n$ inherits a corresponding grading.  The invariant element
$\sigma_i(b)$  has degree $2i$, whereas the element
$\alpha_{i,e}=[a_1\ldots a_{2^i}b_1^e\ldots b_{2^i}^e]$ has degree
$2^i+2e\cdot 2^i=2^i(2e+1)$.  Notice that for every positive integer
$r$ the set $\{\alpha_{i,e}\,|\,
0\leq i,0\leq e\}$ has exactly one element of degree $r$.

\begin{prop}
\label{pr:stable-iso}
The map $\Lambda(\alpha_{i,e}\,|\,0\leq i, 0\leq e) \tens
\F_2[\sigma_i\,|\, i\geq 0] \ra L_\infty$ is an isomorphism.
\end{prop}

\begin{proof}
We have already proven in Corollary~\ref{co:Ln-gen} that the map is a
surjection.  The injectivity will be deduced from a counting argument.
For convenience, let $D$ denote the domain of the map from the
statement of the proposition.
Let $S=\F_2[v_1,v_2,\ldots]$ where $v_i$ has degree $i$.  We will
prove that the Poincar\'e series for $D$ and $L_\infty$ both coincide
with the Poincar\'e series for $S$.  Since $D$ and $L_\infty$ will
therefore have identical Poincar\'e series, the surjection $D\fib
L_\infty$ must in fact be an isomorphism.

Note that $S$ has a basis over $\F_2$ consisting of monomials
\[ v_{i_1}v_{i_2}\cdots v_{i_r} v_1^{2e_1}v_2^{2e_2}\cdots
v_k^{2e_k}\] with each $e_j\geq 0$, where the $i_u$'s are distinct.
There is an evident bijection between the elements of this basis and
the basis for $D$ consisting of monomials in the $\alpha_{i,e}$'s and
$\sigma_i$'s: we replace each $v_{i_r}$ with the unique $\alpha_{i,e}$
having degree $i_r$, and we replace each $v_i^{2e}$ with $\sigma_i^e$
This
identifies the Poincar\'e series for $S$ and $D$.

Recall that $L_\infty$ has a $\Z/2$-basis consisting of the invariants
$[a_{i_1}\ldots a_{i_r}b_{j_1}^{e_1}\ldots b_{j_s}^{e_s}]$ where there
is allowed to be overlap between the $i$- and $j$-indices.    Say that
a monomial is \dfn{pure} if it only contains $a$'s and $b$'s of a
single index.  So $b_i^{e}$ and $a_ib_i^e$ are pure, but $a_1a_2b_1^2$
is not.  
An arbitrary 
monomial $m$ can be written
uniquely (up to permutation of the factors) as 
\[ m=m_1\cdot m_2 \cdots m_t \]
where each $m_i$ is pure and the indices appearing in $m_i$ and
$m_j$ are different for every $i\neq j$.  For example,
\begin{myequation}
\label{eq:monomial}
 a_1a_2a_3a_4 b_1^4b_2b_4b_5^2 = (a_1b_1^4)\cdot (a_2b_2) \cdot
(a_3)\cdot (a_4b_4)
\cdot (b_5^2).
\end{myequation}
For a pure monomial $m$, let $d(m)$ be its degree and let
$\eta(m)=v_{d(m)}$.  Finally, for an arbitrary monomial $m$ as above
define
$\eta(m)=\eta(m_1)\cdots\eta(m_t)=v_{d(1)}\cdot v_{d(2)} \cdots v_{d(t)}$.  
For example, for the monomial in (\ref{eq:monomial}) we have
$\eta(m)=v_1 v_3^2  v_9  v_{10}$.  

Note that if $\sigma$ is a permutation
of the indices then $\eta(\sigma m)=\eta(m)$.  One readily checks that the
function $\eta$ gives a bijection between our basis for $L_\infty$ and
the standard monomial basis for $S$; it should be enough to see the
inverse in one example, e.g.
\[ v_1^3v_2^2v_3^2v_6v_{10} = \eta([a_1a_2a_3\cdot b_4b_5\cdot
a_6b_6a_7b_7\cdot b_8^3\cdot b_9^5]).
\]
Clearly $\eta$ preserves the
homogeneous degrees of the elements, so the Poincar\'e series for
$L_\infty$ and $S$ coincide.  This completes our proof.
\end{proof}

\begin{lemma}
\label{le:iso}
The surjections $L_{n+1}\fib L_n$ and $L_\infty \fib L_n$ are isomorphisms in
degrees less than or equal to $n$.
\end{lemma}

\begin{proof}
This is clear from our description of the additive basis for $L_n$.
\end{proof}

\begin{proof}[Proof of Theorem~\ref{th:En}]
We have already proven (c) in Proposition~\ref{pr:stable-iso}, so it only
remains to prove (a) and (b).  For (a) we have proven in
Corollaries~\ref{co:Ln-gen} and \ref{co:Ln-gen2} that $L_n$ is
generated by the given classes, so we need only show that those
classes are independent modulo $I^2$.  However, all of the classes in
question are in degrees less than $n$.  If there were a relation among
them in $L_n$, this relation would lift to $L_\infty$ by
Lemma~\ref{le:iso}.  Yet in $L_\infty$ the classes are obviously
independent modulo $I^2$.

Finally, we prove (b).  In our list of indecomposables there are $n$
of the form $\sigma_i(b)$ ($1\leq i\leq n$).   The ones of the form
$[a_1\ldots a_{2^i}]$ number $\lfloor \log_2(n)\rfloor$ since we must
have $2^i\leq n$.  The ones of the form $[a_1b_1^e]$ number $\lfloor
n-1\rfloor$,   the ones of the form $[a_1a_2b_1^eb_2^e]$ number $\lfloor
\frac{n}{2}-1\rfloor$, etc.
So we have the formula
\[ \#(\text{indecomposables in $L_n$}) =
n + \lfloor \log_2(n)\rfloor + (n-1) + \lfloor \tfrac{n}{2}-1\rfloor + \lfloor
\tfrac{n}{4}-1\rfloor + \cdots
\]
where the series stops when $\frac{n}{2^i}$ becomes smaller than $1$.
Thus, excluding the first two terms we have $\lfloor \log_2(n)\rfloor$
terms, all of which have a ``-1'' in them.  These negative ones
together cancel the $\lfloor \log_2(n)\rfloor$ term, leaving
\[  \#(\text{indecomposables in $L_n$}) =
2n  + 
\lfloor \tfrac{n}{2}\rfloor + \lfloor
\tfrac{n}{4}\rfloor + \cdots
\]

Let $\alpha(n)=\lfloor \tfrac{n}{2}\rfloor + \lfloor
\tfrac{n}{4}\rfloor + \cdots$.  We complete the proof of (b) by
showing that 
\[ \alpha(n)=n-(\text{number of ones in the binary expansion of $n$}).
\] 
We do this by induction on $n$, the case $n=1$ being trivial.  For the
general case write $n=2^k+n'$ where $n'<2^k$.  Then 
\begin{align*}
 \alpha(n)=(2^{k-1}+2^{k-2}+\cdots + 1) + \alpha(n') = 2^k-1 +
\alpha(n') &= n-n'-1 + \alpha(n') \\
&= n-(n'-\alpha(n')+1).
\end{align*}
By induction, $n'-\alpha(n')$ is the number of ones in the binary
expansion of $n'$---which is also one less than the number in the binary
expansion of $n$.  This completes the proof.
\end{proof}

\bibliographystyle{amsalpha}

\end{document}